\documentclass[a4paper,twoside,10pt]{article}
\usepackage{amsmath,amsthm,amssymb, enumerate}
\usepackage{epsfig, float, afterpage,array,multicol,eucal, amscd,url}
\swapnumbers

\theoremstyle{plain}
\newtheorem{Thm}{Theorem}[section]
\newtheorem{Prop}[Thm]{Proposition}
\newtheorem{Lem}[Thm]{Lemma}
\newtheorem{Cor}[Thm]{Corollary}

\newtheorem*{Thm*}{Theorem}
\newtheorem*{Thm5*}{The Local-indicability Cohen-Lyndon Theorem}

\newtheorem*{Thm6*}{The Cohen-Lyndon Theorem}
\newtheorem{Thm5}[Thm]{The Local-indicability Cohen-Lyndon Theorem}

\newtheorem*{Thm3*}{The Edjvet-Howie Theorem}

\theoremstyle{remark}

\theoremstyle{definition}

\newtheorem{Rem}[Thm]{Remark}
\newtheorem{Ex}[Thm]{Example}

\newtheorem{Rems}[Thm]{Remarks}

\newtheorem{Not}[Thm]{Notation}

\newtheorem{Defs}[Thm]{Definitions}

\newtheorem{Sett}[Thm]{Setting}

\numberwithin{figure}{Thm}

\newcommand{\abs}[1]{\left\lvert#1\right\rvert} 

\newcommand{\gen}[1]{\langle\mkern3mu#1\mkern3mu\rangle}
\newcommand{\normgen}[1]{{\langle}{\mkern-3.7mu}{\lvert}{\mkern1.2mu}{#1}{\mkern1.2mu}{\rvert}{\mkern-3.7mu}{\rangle}}
\newcommand{\gp}[2]{\gen{{#1}\mid #2}}

\newcommand{\lsup}[2]{{}^{#1}\mkern-1mu{#2}}

\renewcommand{\le}{\leqslant}
\renewcommand{\ge}{\geqslant}

\def\d1{\discretionary{-}{}{-}}
\def\maxx{\text{max}}

\def\Z{\mathbb{Z}}

\def\R{\mathbb{R}}

\def \leftmod {{\setminus}}

\def\coloneq{\mathrel{\mathop\mathchar"303A}\mkern-1.2mu=}

\DeclareMathOperator{\glue}{glue}
\DeclareMathOperator{\axis}{axis}
\DeclareMathOperator{\Eaxis}{\text{$E$}axis}

\begin{document}

\pagestyle{myheadings}
\markboth{The Local-indicability Cohen-Lyndon Theorem}{Y. Antol\'{\i}n, W. Dicks and P.\ A.\ Linnell}
\title{On the Local-indicability Cohen-Lyndon Theorem}

\author{  Yago Antol\'{\i}n, Warren Dicks and  Peter A.\ Linnell}

\date{\footnotesize\today}

\maketitle

\begin{abstract}
For a group $H$ and a  subset $X$  of $H$, we let $\lsup{H}{\mkern-2.6mu X}$ denote the set
$\{hxh^{-1} \mid h \in H,\, x \in X\}$, and when
  $X$ is a free\d1generating set of  $H$,  we say that the set $\lsup{H}{\mkern-2.6mu X}$
 is   a \textit{Whitehead  subset of}  $H$.

For a group $F$ and an element $r$ of  $F$, we say that $r$ is \textit{Cohen-Lyndon aspherical in $F$}
if   $ \lsup{F}{\mkern-2mu  \{r\}}$   is a Whitehead subset
of the  subgroup of
$F$ that is generated by $ \lsup{F}{\mkern-2mu  \{r\}}$.

In 1963,  D.\ E.\ Cohen and R.\ C.\ Lyndon   independently showed  that
in each free group  each non-trivial element  is Cohen-Lyndon aspherical.
Their  proof used  the celebrated induction method
devised  by W.~Magnus  in 1930 to study one-relator groups.

In 1987, M.\ Edjvet and J.\ Howie  showed
that if $A$~and $B$ are locally indicable groups,
then each cyclically reduced element of $A{\ast}B$ that does not lie in  $A{\cup}B$
is Cohen-Lyndon aspherical in $A{\ast}B$.
Their  proof used  the Cohen-Lyndon Theorem.

Using Bass-Serre Theory and the Edjvet-Howie Theorem, one can deduce
the Local-indicability Cohen\d1Lyn\-don Theorem:
 if
 $F$ is a locally indicable group  and $T$~is an $F$-tree  with trivial edge stabilizers,
then each element of $F$ that fixes no vertex of $T$
is Cohen-Lyndon aspherical in  $F$.
Conversely, the Cohen-Lyndon Theorem and the Edjvet-Howie Theorem
are immediate consequences of the Local-indicability Cohen-Lyndon Theorem.

In this article, we give a detailed   review of  Howie induction
and  arrange the arguments of  Edjvet and Howie  into  a Howie-inductive proof of
the Local-indicability Cohen-Lyndon Theorem that  does not use Magnus induction
or the Cohen-Lyndon Theorem.

We  conclude  with a  review of some standard applications
of Cohen-Lyndon asphericity.

\medskip

{\footnotesize

}
\end{abstract}

\maketitle

\section{Outline}

\begin{Not}
Let $F$ be a multiplicative group, fixed throughout the article.

The disjoint union of two sets $X$ and $Y$   will be denoted by~$X \vee Y$.

By a \textit{transversal for} a (left or right) $F$-action on a set $X$  we mean
a subset  of $X$ which contains exactly one element of each $F$-orbit of $X$;
by the axiom of choice, transversals always exist.

For elements $x$, $y$ of $F$,  we write $\overline x \coloneq x^{-1}$,  $\lsup{x}{y} \coloneq xy\overline x$,
$[x,y]\coloneq xy\overline x \,\overline y$,
and
\mbox{$\mathbf{C}_F(x) \coloneq \{f \in F \mid \lsup{f}{x}=x\}$.}
For any subgroup $H$ of $F$,  we say that  \mbox{$x$, $y$} are    \textit{$H$-conjugate}
if there exists some $h \in H$ such that $\lsup{h}{x} = y$.
Conjugation actions will always be left actions in this article.

If $R$ and $X$ are subsets  of $F$, we let $\gen{R}$ denote the subgroup of $F$ generated by $R$,
and we write
$\lsup{X}{\mkern-2mu R}\coloneq \{\lsup{x}{r} \mid r \in R, x \in X\}$
 and $F/\normgen{R} \coloneq F/\gen{\lsup{F}{\mkern-2mu R}}$.
If $R$ consists of a single element~$r$, we write simply $\gen{r}$, $\lsup{X}{r}$ and $F/\normgen{r}$, respectively.
We say that $X$   is   a
\textit{free-generating set of~$F$}
if the induced group homomorphism   \mbox{$\gp{X}{\quad} \to F$} is   bijective.

 A subset $Y$ of $F$ is said to be a \textit{Whitehead subset of}  $F$ if
there exists some free-generating set $X$ of $F$ such that  $\lsup{F}{\mkern-2.6mu X} = Y$,
that is, $Y$ is closed under
the $F$-conjugation action   and some transversal for the $F$-conjugation action  on $Y$  is
a free-generating set of $F$.

Borrowing terminology from~\cite{CCH}, we say that  an element  $r$ of $F$ is   \textit{Cohen-Lyndon aspherical in $F$}
 if $ \lsup{F}{\mkern-2mu r}$ is a Whitehead subset of $ \gen{\lsup{F}{\mkern-2mu r}} $.
Thus, $r$  is    Cohen-Lyndon aspherical in $F$  if and only if
some  transversal   for the  $ \gen{\lsup{F}{\mkern-2mu r}}$-conjugation action on $ \lsup{F}{\mkern-2mu r}$
 is  a free-generating set  of~$ \gen{\lsup{F}{\mkern-2mu r}} $.  Thus,
 $r$  is    Cohen-Lyndon aspherical in~$F$  if and only if there exists some subset $X$ of $F$
such that $X$ is a transversal  for the
$(\gen{\lsup{F}{\mkern-2mu r}}\mathbf{C}_F(r))$-action on~$F$   by
multiplication on the right
and, moreover, $\lsup{X}{r}$ is a free-generating set of~$\gen{\lsup{F}{\mkern-2mu r}}$.
Here, $X \ne \emptyset$ and $r \ne 1$ and the map $X \to\lsup{X}{\mkern-2mu r}$, $x \mapsto \lsup{x}{r}$, is
bijective.

We say that a subset $R$ of $F$  is  \textit{Cohen-Lyndon aspherical in $F$}
 if  no two distinct elements of $R$ are $F$-conjugate and
 $ \lsup{F}{\mkern-2mu R}$ is a Whitehead subset of $ \gen{\lsup{F}{\mkern-2.6mu R}} $.
Thus, $R$  is    Cohen-Lyndon aspherical in $F$  if and only if
no two distinct elements of $R$ are $F$-conjugate and
 some
transversal for the  $ \gen{\lsup{F}{\mkern-2.6mu R}}$-conjugation action  on $ \lsup{F}{\mkern-2.6mu R} $
 is  a free-generating set of $\gen{\lsup{F}{\mkern-2.6mu  R}}$.
 Thus, $R$  is    Cohen-Lyndon aspherical in $F$  if and only if
no two distinct elements of $R$ are $F$-conjugate and there exists
some family $(\, X_r \mid r \in R)$ of subsets of $F$ with the properties that, for each $r \in R$, $X_r$ is a
transversal for the $(\gen{\lsup{F}{\mkern-2mu R}}\mathbf{C}_F(r))$-action
on~$F$ by
multiplication on the right  and, moreover,
$\mathop{\bigcup}\limits_{r\in R} (\lsup{X_r\,}{r})$ is  a free-generating set of~$\gen{\lsup{F}{\mkern-2mu R}}$.

A group is  said to be  \textit{indicable} if it
  is  trivial or it has some quotient that is
 infinite and cyclic.
A group is  said to be \textit{locally indicable} if all its finitely generable subgroups are  indicable.
For example, all free groups are locally indicable.
\hfill\qed
\end{Not}

In~\cite[Theorem~4.1]{CohenLyndon}, D.\ E.\ Cohen and R.\ C.\ Lyndon proved the following.

\begin{Thm6*}  If $F$ is a free group,
then each   non-trivial element of~$F$ is Cohen-Lyndon aspherical in $F$.
\hfill\qed
\end{Thm6*}

\noindent Both the  proof  in~\cite{CohenLyndon}
and  its simplification by
A.\ Karrass and D.\ Solitar~\cite[Theorem~2]{KarrassSolitar}  use the famous induction method
that was devised by W.\ Magnus in 1930 to study one-relator groups.

In~\cite[Theorem~1.1]{EdjvetHowie}, M.\ Edjvet and J.\ Howie
used
 the Cohen-Lyndon Theorem   and  Karrass-Solitar reduction
 to prove  the following.

\begin{Thm3*} If $A$~and $B$ are locally indicable groups,
then each cyclically reduced element of $A{\ast}B$ that does not lie in  $A{\cup}B$
is Cohen-Lyndon aspherical in $A{\ast}B$.
\hfill\qed
\end{Thm3*}

By applying Bass-Serre Theory, one can deduce the following.

\begin{Thm5*}  If $F$ is a locally indicable group
and   $T$ is an   $F$-tree   with trivial edge stabilizers, then each
  element of $F$ that fixes no vertex of $T$   is Cohen-Lyndon aspherical in~$F$.
\hfill\qed
\end{Thm5*}

\begin{Ex}\label{Ex:free}  Let  $X$ be a set and let $F = \gp{X}{\quad}$.

Let $T$ be the $F$-graph   with  vertex set  $F$
and  edge set   $F{\times}X$ such that  each edge \mbox{$(f,x) \in  F{\times}X $} has initial vertex $f$
and terminal vertex $fx$.  In a natural way,   $F$ acts  on~$T$.  The stabilizers are trivial.
It is well known that $T$ is a tree; see, for example,~\cite[Theorem~I.7.6]{DicksDunwoody89}.

Thus, the  Cohen-Lyndon Theorem  is
the case of the Local-indicability Cohen-Lyndon Theorem
 where $F$ acts freely on $T$.
\hfill\qed
\end{Ex}

\begin{Ex}\label{Ex:freeproduct} Let $A$ and $B$ be locally indicable groups and let $F = A {\ast}B$.

Let $T$ be the $F$-graph with  vertex set
$(F/A) \vee (F/B)$
and  edge set   $F$ such that  each edge $f \in  F $ has initial vertex $fA$
and terminal vertex $fB$.
In a natural way $F$ acts  on $T$.  The edge stabilizers are trivial and
the elements of $F$ which  fix  vertices of $T$ are the $F$-conjugates of the elements of $A{\cup}B$.
It can be shown that $T$ is a tree; see, for example,~\cite[Theorem~I.7.6]{DicksDunwoody89}.

Thus, the Edjvet-Howie Theorem  is
the case of the Local-indicability Cohen-Lyndon Theorem  where $T$ has only one edge $F$-orbit and two
vertex $F$-orbits. \hfill\qed
\end{Ex}

In his work on locally indicable groups in~\cite{Howie81},~\cite{Howie82},~\cite{Howie00},
Howie developed a powerful induction technique that
amounts to being given certain information about a group,
\begin{list}{}{}
\item and then choosing an appropriate finitely generable subgroup which contains the given information,
\item and then choosing an appropriate  normal subgroup of the finitely generable  subgroup
for which the quotient group is infinite and cyclic,
and translating the given information into information about the
normal subgroup,
\end{list}
\noindent and repeating this two-step cycle  as often as possible.  This  simple procedure  has many applications.
Magnus induction has a formally similar format but requires a  free-generating set
and more careful choices at each step.

In this article, we give a   Bass-Serre-theoretical survey of Howie induction
and  re-arrange the arguments used by  Edjvet and Howie in~\cite{EdjvetHowie}
into a Howie-inductive proof of
the Local-indicability Cohen-Lyndon Theorem   that does not use Magnus induction
or the Cohen-Lyndon Theorem.

In Section~\ref{sec:summary}, we introduce   definitions  concerning staggerable subsets, strongly staggerable subsets,
and other concepts
that we shall be using.

In Section~\ref{sec:3}, we describe the
finite descending chain of subgroups of~$F$ used in Howie induction.
 We then see that the staggerable conditions can be moved all the way down
the chain.

In Section~\ref{sec:4},  we find that at the bottom of any of the  chains of Section~\ref{sec:3} the conditions of
staggerability and local indicability
interact to produce Whitehead subsets.  This section reproduces
parts of Appendix~A of~\cite{ADL} with some modifications to suit the current applications.

In Section~\ref{sec:5}, we recall some technical results of Cohen-Lyndon, of Karrass-Solitar, and of Edjvet-Howie.
These results then allow  us to take information
that was spontaneously generated at the bottom of the chain and move it all the way back up the chain;
the Local-indicability Cohen-Lyndon Theorem then follows.

Corollary~2.2 of~\cite{CohenLyndon} gives a sufficient condition for a  subset of a free group~$F$  to be
Cohen-Lyndon aspherical in   $F$.  In  Section~\ref{sec:6},  we find that the preceding machinery implies
the more general result that,
for any  locally indicable group $F$ and   any   subset $R$ of $F$,  if there exists
some $F$-tree $T$ with trivial edge stabilizers
such that $R$ is strongly
$T$-staggerable modulo~$F$, then $R$
is Cohen-Lyndon aspherical in $F$.   Standard arguments then
yield consequences concerning the quotient group  $G \coloneq F/\normgen{R}$;
for example,  one obtains information about  the  torsion subgroups of $G$ and the
higher homology groups of~$G$.

\section{Staggerability} \label{sec:summary}

In this section, we introduce   definitions  concerning staggerability and other concepts
that we shall be using.
Recall that $F$ is a   multiplicative group.

\begin{Not}
For any   set $X$,  we let  $\abs{X}$ denote the  cardinal of $X$,
and we let $\Z X $ and $\Z[X]$  denote the  $\Z$-module  that is free on $X$.

By an \textit{ordering}  of a set $X$, we shall mean a  binary relation  which \textit{totally} orders
$X$.

We will find it useful to have a notation  for intervals in $\Z$ that is different from the notation for
intervals in $\R$.
Let $i$, $j \in \Z$.
We write $[i{\uparrow}j]\coloneq \{k \in \Z \mid   k \ge i\text{ and } k \le j\}$,
and $\left]-\infty{\uparrow}j\right] \coloneq \{k \in \Z \mid  k \le j\}$, and
 $[i{\uparrow}\infty[\,\,\,\, \coloneq \{k \in \Z \mid k \ge i \}$.

We shall define families of
subscripted symbols by using the following convention.
Let $v$ be a symbol.
For each $k \in \Z$, we let $v_k$ denote the ordered pair
$(v,k)$, and, for each subset $I$ of $\Z$, we let $v_I \coloneq  (v_k \mid k \in I)$.

For two subsets $Y$, $Z$ of a set $X$, the complement of $Y \cap Z$ in $Y$ will be denoted
by $Y-Z$ (and not by $Y\setminus Z$ since we let
 $F\backslash Z$ denote the set of $F$-orbits of a left $F$-set $Z$).

For any subset $Y$ of a left $F$-set $X$, we write
$\glue(F,Y) \coloneq \{f\in F \mid f\,Y \cap  Y \ne \emptyset\}$.

We write $F' \coloneq \gen{\{[x,y]\mid x,y\in F\}} \unlhd F$ and  $F^{\text{ab}} \coloneq
F/F'$, the abelianization of $F$.
\hfill\qed
\end{Not}

\begin{Defs}
 Let  $r$ be an  element of $F$.

We say that $r$  \textit{has a unique root in $F$}  if  $r=1$  or
$r$ lies in a unique maximal infinite cyclic  subgroup of $F$.

If $r \ne 1$ and $r$ lies in a unique maximal infinite cyclic subgroup $C$ of $F$,
we define \textit{the unique root of $r$ in~$F$}, denoted by $\sqrt[F]{r}$,
to be the unique generator  of $C$ of which
$r$ is a positive power.  We say that $1$  \textit{is the unique root of} $1$,
and we define $\sqrt[F] 1 \coloneq  1$.

If $r \ne 1$ and $\mathbf{C}_F(r)$ is infinite and cyclic, then $r$ has a unique root in $F$ and $\gen{\mkern-6mu \sqrt[F]{r}} = \mathbf{C}_F(r)$.

We say that a subset  $R$ of $F$ \textit{has unique roots in $F$} if every element of $R$ has a unique root in $F$,
in which case  we let $\sqrt [F]R$ denote the set of these unique roots in $F$.
\hfill\qed
\end{Defs}

We now define the staggerability concepts we need.
We shall use~\cite{DicksDunwoody89} as our reference  for Bass-Serre Theory.

\begin{Defs}\label{defs:tree} Let $T = (T,VT,ET, \iota,\tau)$ be an $F$-tree and suppose that $ET$ is $F$-free.

\medskip

\noindent\mbox{\phantom{iiv}}(i). Let $r$ be an element of $F$   that   fixes no vertex of $T$.

There exists a  unique  minimal $\gen{r}$-subtree of~$T$, which is denoted by $\axis(r)$ and
has the form of a real line shifted by $r$; see,
for example,  \mbox{\cite[Proposition~I.4.11]{DicksDunwoody89}.}
We write $\Eaxis(r)\coloneq E(\axis(r))$, and we shall be particularly interested in the finite set
$ F\leftmod (F(\Eaxis(r))) \coloneq    \{Fe  \mid e \in \Eaxis(r)\} \subseteq   F\leftmod ET.$

For all $f \in F$, $\axis(\lsup{f}{r}) = f\axis(r)$ and
$F\leftmod (F(\Eaxis(\lsup{f}{r}))) = F\leftmod (F(\Eaxis(r))) $.

For all $n \in \Z - \{0\}$\,,  $\axis(r^n) = \axis(r)$.

Since  $F$ acts freely on $ET$,  it can be shown that $\mathbf{C}_F(r)$  acts freely on
$\axis(r)$, and, by \mbox{Bass-Serre Theory,} $\mathbf{C}_F(r)$ is infinite and cyclic.  Here,
 $r$ has a unique root in $F$ and
\mbox{$\gen{r} \subseteq \gen{\mkern-6mu\sqrt[F]{r}} =\mathbf{C}_F(r)  \subseteq \glue(F, \Eaxis(r))$;}
notice that if $\glue(F, \Eaxis(r)) = \gen{r}$, then   $\sqrt[F]{r} = r$.

\medskip

\noindent\mbox{\phantom{iv}}(ii).  If $<$ is some (total) ordering   of $ET$ and $R$ is some subset of $F$,
 we say that  $R$   is  \textit{$(T,{<})$-staggered modulo~$F$}
if the following three conditions hold.
\begin{enumerate}[$\mkern50mu$(S1).]
\item Each $F$-orbit in $ET$ is an interval in $(ET, <)$; here,
 there  exists a  unique  ordering of  $F\leftmod ET$, again denoted by $<$, with the property that
for all $e$, $e'\in ET$, $Fe < Fe'$ if and only if   $e < e'$ and $Fe \ne Fe'$.
 \item  Each element of $R$ fixes no vertex of~$T$.
\item For each  $(r_1,r_2) \in R{\times}R$,
exactly one of the following three conditions holds.
\begin{enumerate}[(a)]
\item $\lsup{F}{\mkern-2mu r}_1 = \lsup{F}{\mkern-2mu r}_2$.
\item In $(F\leftmod ET, <)$, \newline
$\min (F\leftmod (F(\Eaxis(r_1)))) \,\,  < \,\, \min(F\leftmod (F(\Eaxis(r_2)) ))$  and \newline
$\max (F\leftmod (F(\Eaxis(r_1)))) \,\, < \,\, \max(F\leftmod (F(\Eaxis(r_2)) ))$.
\item In $(F\leftmod ET, <)$,   \newline
$\min (F\leftmod (F(\Eaxis(r_2)))) \,\,  < \,\, \min(F\leftmod (F(\Eaxis(r_1))) )$  and \newline
$\max (F\leftmod (F(\Eaxis(r_2)))) \,\, < \,\, \max(F\leftmod (F(\Eaxis(r_1))) )$.
\end{enumerate}
\end{enumerate}

We say that $R$  is  \textit{strongly $(T,{<})$-staggered modulo~$F$} if, moreover,
the following four conditions hold.
\begin{enumerate}[$\mkern50mu$(S1).]
\setcounter{enumi}{3}
\item No two distinct elements of $R$ are $F$-conjugate.
\item $F\leftmod T$ has a unique maximal subtree.
\item $(F\leftmod ET, <)$ is order isomorphic to an interval in $\Z$.
\item Any two $<$-consecutive edges in $(F\leftmod ET, <)$ have a vertex in common in  $F\leftmod T$.
\end{enumerate}

We have two types of examples  in mind satisfying (S5), (S6) and (S7):
 cases where $F\leftmod T$ has only one vertex  and   cases where
 $F\leftmod T$ has the form of the real line.

\medskip

\noindent\mbox{\phantom{v}}(iii). A subset $R$ of $F$ is said to be \textit{$T$-staggerable modulo $F$}
 if
there exists some ordering $<$  of $ET$
such that $R$ is $(T,{<})$-staggered modulo $F$.

Notice that if $R$ is  $T$-staggerable modulo~$F$, then $R$ has unique roots in~$F$.

Notice that if $\abs{R}=1$, then $R$ is  $T$-staggerable modulo~$F$ if and only if
 the unique element of $R$ fixes no vertex of $T$.

A subset $R$ of $F$ is said to be \textit{strongly $T$-staggerable modulo $F$}
 if
there exists some ordering $<$  of $ET$
such that $R$ is strongly $(T,{<})$-staggered modulo $F$.
\hfill\qed
\end{Defs}

\begin{Rems}
In the case where $F$ is free, the concept of a staggered subset appeared in Howie's \mbox{article~\cite[p.642]{Howie00}}
and was a generalization, from the case where $\abs{F\leftmod VT} = 1$ and $F$ is free,
of a definition  of a staggered presentation  that is given in
 Lyndon and Schupp's book~\cite[p.152]{LyndonSchupp}.
In the case where $\abs{F\leftmod VT} = 1$ and $F$ is free, the above
concept of a strongly staggered subset
corresponds to a definition that is given in \mbox{\cite[ p.104]{LyndonSchupp}}
of what could reasonably  be called a  strongly staggered presentation
but is somewhat confusingly called a ``staggered presentation";  such  strongly staggered presentations
 arise unnamed in the hypotheses of Lyndon's (Non-Simple) Identity Theorem~\cite[Section~7]{Lyndon} and
Cohen and Lyndon's Corollary~2.2~\cite{CohenLyndon}.

In Corollary~\ref{Cor:big} below, we show that, for any locally indicable group $F$ and any subset $R$  of $F$,
if there exists some $F$-tree $T$
with trivial edge stabilizers such that $R$  is strongly
$T$-staggerable modulo~$F$, then $R$ is  Cohen-Lyndon aspherical in $F$.  The case of  Corollary~\ref{Cor:big}
where $\abs{F\leftmod VT} = 1$  and $F$ is free  is precisely
the Cohen-Lyndon result for  strongly staggered presentations~\mbox{\cite[Corollary~2.2]{CohenLyndon};}
see  \mbox{also~\cite[Proposition III.11.1]{LyndonSchupp}.}
\hfill\qed
\end{Rems}

We shall be using the following well-known observation.

\begin{Rem}\label{Rem:tree}
Let $X$ be a generating set of $F$, let $T$ be an $F$-tree,
and let $Y$ be a subtree of $T$ such that, for some vertex $v$ of $T$, $Y \supseteq \{v \} \cup Xv $, or,
more generally, such that \mbox{$X \subseteq \glue(F,Y)$.}

Consider the $F$-subforest $FY$ of $T$, and let $T'$ denote the component of $FY$ that contains~$Y$.
The  set $\{f\in F \mid fY \subseteq T'\}$
 is closed under right multiplication by elements of $X \cup X^{-1}$ and contains $1$, hence it is all of $F$.
Thus $FY = T'$. Hence $FY$ is an $F$-subtree of~$T$. Moreover, the map $Y \to F \leftmod T'$, $y \mapsto Fy$, is surjective.
\hfill\qed
\end{Rem}

\section{Moving information down a chain}\label{sec:3}

In this section, without mentioning local indicability, we shall see how to
move a staggerable subset $S$ down a finite chain of subgroups
to a finitely generable group whose abelianization is virtually generated by $S$.
The finite descending chain of subgroups is essentially the same chain of subgroups considered by Howie in
his tower arguments in~\cite{Howie82} and~\cite{Howie84}.

We shall often consider the following situation.

\begin{Sett}\label{Sett:1}
Let  $S$ and~$\Phi$  be finite subsets of $F$ such that
$\gen{S{\cup}\Phi}/\normgen{S}$ has no infinite, cyclic quotients.

 Let $T$ be  an $F$-tree with trivial edge stabilizers.

Choose  an arbitrary vertex  $v$    of~$T$.
Let $Y$  be the smallest  subtree  of $T$ which contains \mbox{$\{v\} \cup Sv\cup \Phi v$.}
Then $Y$ is finite.
 Since $F$ acts freely on $ET$, $\glue(F,EY)$ is finite.
We set  \mbox{$S^+ \coloneq S\, \cup \Phi \,  \cup \,\glue(F,EY)$.}  Then $S^+$ is a finite
subset of  $ \glue(F,Y)$. We set $\nu \coloneq {\abs{S^+}+1}$.
For each subgroup $H$ of $F$, we set $H^\dag \coloneq \gen{S^+ \cap H}$.
Notice that if $H$ contains $S \cup \Phi$, then $H^\dag$ also contains $S\cup \Phi$.

We set $F_0\coloneq F$.
Suppose that, for some $n \in [0{\uparrow}(\nu{-}1)]$,
we have  a subgroup $F_n$ \mbox{of $F$} \mbox{containing $S$.}
If  $F_n^\dag /\normgen{S}$ has no infinite, cyclic quotients,
we choose $F_{n+1} \coloneq F_n^\dag$; otherwise, we choose   $F_{n+1}$ to be an arbitrary normal subgroup of  $F_n^\dag $
which contains~$S$ such that   $F_n^\dag/F_{n+1}$ is infinite and cyclic.
After $\nu$ such steps, we   will have  recursively chosen a finite, monotonically decreasing sequence
$F_{[0{\uparrow}\nu]}$
 of subgroups of $F$ that contain  $S$.
\hfill\qed
\end{Sett}

\begin{Prop}\label{Prop:1} In
{\normalfont \mbox{Setting \ref{Sett:1}}},  the following hold.
\begin{enumerate}[{\normalfont (i).}]
\item  $F_{\nu} =\gen{S^+ \cap F_{\nu}} $ and $S \cup \Phi\subseteq F_\nu$ and
$F_{\nu}/\normgen{S}$ has no infinite, cyclic quotients.
\item If $S$ is $T$-staggerable modulo $F$, then $S$ is $T$-staggerable modulo~$F_{\nu}$.
\end{enumerate}
\end{Prop}

\begin{proof} (i).
  If
$n \in [0{\uparrow}(\nu{-}1)]$ and $S \cup \Phi \subseteq F_n$, then
$S \cup \Phi \subseteq F_n^\dag$ and, since $\gen{S \cup \Phi}/\normgen{S}$ has no infinite, cyclic quotients,
 $S \cup \Phi \subseteq F_{n+1}$.
By induction, \mbox{$S \cup \Phi \subseteq   F_\nu$.}

Consider any $n \in [0{\uparrow}(\nu{-}1) ]$.
If $F_{n+1} <  F_n^\dag \,\,\,\,\,(= \gen{S^+ \cap F_n} \le F_n)$, then
$ F_{n+1} \not \supseteq  S^+  \cap F_n $, and, hence, $ S^+ \cap F_{n+1} \subset S^+  \cap F_n$;
notice that the number of those $n$ for which the latter happens is at most $\abs{S^+} = \nu{-}1$.
It follows that there exists some  $\mu \in [0{\uparrow}(\nu{-}1)]$ such that
$F_{\mu+1} = F_\mu^\dag$.  Then $F_\mu^\dag /\normgen{S}$
 has no infinite, cyclic quotients, and
$$F_\mu^\dag  = \gen{S^+ \cap F_\mu}  = \gen{S^+ \cap F_\mu^\dag} = \gen{S^+ \cap F_{\mu+1}}
= F_{\mu+1}^\dag.$$  It then follows that the sequence
$F_{[\mu{\uparrow}\nu]}^\dag$  is constant, and then
 (i) follows.

 (ii). By hypothesis, $S$ is $T$-staggerable modulo  $F_0 \,\,\,\,(= F)$.

Let $n$ be an element of $ [0{\uparrow}(\nu{-}1)]$ such that $S$ is $T$-staggerable modulo $F_n$.
By induction, it suffices to show that $S$ is $T$-staggerable modulo $F_{n+1}$.

There exists some ordering $<$ of $ET$ such that $S$ is $(T,<)$-staggered modulo $F_n$.
By altering the ordering $<$ of $ET$ within each $F_n$-orbit, we can arrange that
every $F_n^\dag$-orbit is an interval with respect to~$<$.
We claim that $S$ is now $(T,<)$-staggered modulo $F_n^\dag$.
It  suffices to show that if  two elements of $S$ are $F_n$-conjugate,
then they are $F_n^\dag$-conjugate.  Suppose then that we have
  $f \in F_n$ and  $r$, $\lsup{f}{r} \in S \subseteq F_n^\dag$.  Now
\mbox{$\axis(r) \subseteq \gen{r}(Y) \subseteq F_n^\dag(Y)$} and
$f\axis(r) = \axis(\lsup{f}{r}) \subseteq \gen{\lsup{f}{r}}(Y) \subseteq F_n^\dag(Y)$.
Hence,
$f \in  \glue(F_n,F_n^\dag(EY))$.  Thus $F_n^\dag fF_n^\dag$ meets
$$\glue(F_n, EY)  \subseteq S^+ \cap F_n \subseteq \gen{S^+ \cap F_n} = F_n^\dag.$$
Hence $f \in F_n^\dag$.  Thus    $S$ is $(T,<)$-staggered modulo $F_n^\dag$.

If $F_{n+1} = F_n^\dag$, then $S$ is    $(T, <)$-staggered modulo $F_{n+1}$ as desired.
Thus, we may assume that
$F_n^\dag/F_{n+1}$ is infinite and cyclic.
There exists some  $z \in F_n^\dag$ such that $zF_{n+1}$ generates $F_n^\dag/F_{n+1}$.
Then $F_n^\dag = \gen{z}F_{n+1}$.  Since $F_n^\dag/F_{n+1}$ is infinite,
 \mbox{$\gen{z} \cap F_{n+1} = \{1\}$.}  Within each $F_n^\dag$-orbit in $ET$, the
$F_{n+1}$-orbits are permuted by $z$ and form  a single $\gen{z}$-orbit.
By altering the ordering $<$ of $ET$ within each $F_n^\dag$-orbit, we can arrange that
every $F_{n+1}$-orbit is an interval with respect to~$<$, and such that
$z$ moves each $F_{n+1}$-orbit to the  next $<$-largest orbit.
It suffices to examine two elements of $S$ which are $F_n^\dag$-conjugate.
 Suppose then that we have
  $f \in F_{n+1}$ and $i \in \Z$ and $r \in S$ such that  $\lsup{z^if}{r} \in S$.
If $i=0$, then $r$ and $\lsup{z^if}{r}$ are $F_{n+1}$-conjugate, which is
one of the desired possibilities. We  now  assume that $i > 0$; the case where $i < 0$ is similar.
Now
$$F_{n+1}\!\leftmod (F_{n+1}(\Eaxis(\lsup{z^if}{r}))) = F_{n+1}\!\leftmod (F_{n+1}(z^if\Eaxis(r)))
=   z^i(F_{n+1}\!\leftmod ( F_{n+1}(\Eaxis(r))).$$
Thus,
$$ \min(F_{n+1}\!\leftmod (F_{n+1}(\Eaxis(\lsup{z^if}{r})) )) = z^i\min (F_{n+1}\!\leftmod( F_{n+1}(\Eaxis(r))))  >
\min (F_{n+1}\!\leftmod (F_{n+1}(\Eaxis(r)))),$$
$$ \max(F_{n+1}\!\leftmod (F_{n+1}(\Eaxis(\lsup{z^if}{r}))) ) = z^i\max (F_{n+1}\!\leftmod (F_{n+1}(\Eaxis(r))) )  >
\max (F_{n+1}\!\leftmod (F_{n+1}(\Eaxis(r)))).$$
Hence, $S$ is   $(T,<)$-staggered modulo $F_{n+1}$, as desired.
\end{proof}

\section{New information generated at the bottom of a chain}\label{sec:4}

In this section, we review, with some changes, the main results of~\cite[Appendix~A]{ADL}.
The following is a minor extension of~\cite[Lemma~A.2.1]{ADL}.

\begin{Lem}\label{lem:bottom} Let $F$ be a finitely generable, locally indicable group,
let $T$ be an $F$-tree with trivial edge stabilizers, and let
$S$ be a subset of $F$ such that $S$ is $T$-staggerable modulo $F$ and  $\sqrt[F]{S}=S$.
Then the following hold.
\begin{enumerate}[{\normalfont (i).}]
\item $F/\normgen{S}$ is indicable.
\item If $F/\normgen{S}$ is trivial, then, for each $r \in S$,
$\glue(F, \Eaxis(r)) =  \gen{r}$, and
$F$ acts freely on~$T$, and   $\lsup{F}{\mkern-2mu S}$  is a  Whitehead subset of $F$.
\end{enumerate}
\end{Lem}

\begin{proof} Recall from Definitions~\ref{defs:tree}
that $S$ has unique roots in $F$.

Without loss of generality, we may replace $S$ with $\lsup{F}{\mkern-2mu S}$, and
we then have $\lsup{F}{\mkern-2mu S} =S$.

By hypothesis, there exists some ordering   $<$ of   $ET$
such that $S$ is $(T, {<})$-stag\-gered modulo~$F$.

Let $X$ be a finite generating set of $F$, let
$v$ be a vertex of $T$, let $Y$ be the smallest subtree of $T$ containing $\{v\} \cup Xv$,
and let $T' \coloneq FY$.  By Remark~\ref{Rem:tree}, $T'$ is an $F$-subtree of $T$ and $F \leftmod T'$ is finite.
Consider any $r \in S$.  Then $T'$ is an $\gen{r}$-subtree of $T$.  Since  $\axis(r)$ is the
unique smallest $\gen{r}$-subtree of $T$, we see that $T'$ contains $\axis(r)$.
It follows that $S$ is $(T',<)$-staggered modulo~$F$.  It now
suffices to prove the result with $T'$ in place of $T$.
Thus, we may assume that $F \leftmod T$ is finite.

By induction, we may assume that the result holds for all smaller values of $\abs{F\leftmod ET}$.

Since  $F$ is a finitely generable, locally indicable group,
$F$  is indicable, and it can be seen that (i) and (ii) hold when $S$ is empty.
Thus we  may assume that $S$ is  non-empty.  It then follows from
 Definitions~\ref{defs:tree} that $\abs{F\leftmod ET} \ge 1$.
In particular,  there exists some
$e_\maxx \in \bigcup\limits_{r  \in S} \Eaxis(r)$
such that,   in $(F\leftmod ET, <)$, $$Fe_\maxx
\quad = \quad \max \{Fe \mid e \in\textstyle \bigcup\limits_{r  \in S} \Eaxis(r)\}.$$
There then exists some $r_\maxx \in S$ such that $e_\maxx \in \Eaxis(r_\maxx)$,
and, by the definition of $(T,{<})$-stag\-gered modulo $F$,
$Fe_\maxx$ does not meet the axis of any element of
$S - \lsup{F}{\mkern-2mu(r_\maxx)}$.
Thus there exists some pair $(r,e)$, for example  $(r,e) = (r_\maxx, e_\maxx)$, such that the following hold.
\begin{align}
&r \in S \text{, } e \in \Eaxis(r) \text{, and }
  Fe \text{ does not meet the axis of any element of } S- \lsup{F}{\mkern-2mu r}. \label{eq:disj}
 \\
&
\text{If }  \glue(F, \Eaxis(r))  \ne  \gen{r}, \text{ then } (r,e) = (r_\maxx,e_\maxx).\label{eq:max}
\end{align}

We consider the $F$-forest $T - Fe$.  Let $T_\iota$ denote the component of $T - Fe$ containing
$\iota e$, and let $T_\tau$ denote the component  of $T - Fe$ containing $\tau e$.
Let $F_\iota$ denote the $F$-stabilizer of $\{T_\iota\}$,
and let   $F_\tau$ denote the $F$-stabilizer of $\{T_\tau\}$, where we are using set brackets to
emphasize that we want to consider each component of $T - Fe$  as a single element.
Let $S_\iota \coloneq S \cap F_\iota$ and $S_\tau \coloneq S \cap F_\tau$.
By~\eqref{eq:disj}, for each $r' \in S -  \lsup{F}{\mkern-2mu r} $, $\axis(r')$ lies in $T - Fe$ and hence lies in a
component of $T - Fe$.  It follows that $\lsup{F}{\mkern-2mu S}_\iota
\,\cup \,\lsup{F}{\mkern-2mu S}_\tau \, \cup\,  \lsup{F}{\mkern-2mu r}$ is all of $S$.
Notice that if $F\{T_\iota\} = F\{T_\tau\}$ then  $\lsup{F}{\mkern-2mu S}_\iota  = \lsup{F}{\mkern-2mu S}_\tau$.

By applying the Bass-Serre Structure Theorem to
 the $F$-tree whose vertices are the components of $T - Fe$
and whose edge set is $Fe$, with $fe$ joining $f T_\iota$ to $f T_\tau$ $(f \in F)$, we find that
\begin{equation}\label{eq:preconn}
F = \begin{cases}
F_\iota {\ast} F_\tau &\text{if $F\{T_\iota\} \ne F\{T_\tau\}$,}\\
F_\iota {\ast} \gp{f}{\quad} &\text{if  $f \in F$ and $f T_\iota  =   T_\tau $.}
\end{cases}
\end{equation}
Hence,
\begin{equation}\label{eq:conn}
F/\normgen{S-\lsup{F}{\mkern-2mu r}} =
\begin{cases}
(F_\iota/\normgen{S_\iota}) {\ast} (F_\tau/\normgen{S_\tau}) &\text{if $F\{T_\iota\} \ne F\{T_\tau\}$,}\\
(F_{\iota}/\normgen{S_{\iota}}) {\ast}  \gp{f}{\quad}  &\text{if  $f \in F$ and $f T_\iota  =   T_\tau $.}
\end{cases}
\end{equation}

\medskip

\newpage

\noindent\textbf{Case 1.} Both
$F_\iota/\normgen{S_\iota}$ and $F_\tau/\normgen{S_\tau}$ have
infinite, cyclic quotients.

By~\eqref{eq:conn},
$F/\normgen{S-\lsup{F}{\mkern-2mu r}}$ has a rank-two, free-abelian quotient.
On incorporating $r$, we see that $F/\normgen{S}$ has an
infinite, cyclic quotient, and, hence, (i) and (ii)   hold in this case.

\medskip

\noindent\textbf{Case 2.} One of
$F_\iota/\normgen{S_\iota}$,  $F_\tau/\normgen{S_\tau}$  does not have an
infinite, cyclic quotient.

By reversing the orientation of every edge of $T$, if necessary,
we may assume  that $F_\iota/\normgen{S_\iota}$ does not have an infinite,
cyclic quotient.

By the induction hypothesis applied to $(F_\iota, T_\iota, S_\iota)$,
we see that $F_{\iota}/\normgen{S_{\iota}}$ is trivial,
and that $F_{\iota}$ acts freely on $T_{\iota}$, and that some
transversal for that $F_{\iota}$-action on $ S_{\iota} $ by conjugation
is a (finite) free-generating set of~$F_{\iota}$, and that,
for each $r_{\iota} \in S_{\iota}$, $\glue(F_\iota, \Eaxis(r_\iota)) =  \gen{r_\iota}$,
and, hence, \mbox{$\glue(F, \Eaxis(r_\iota)) =  \gen{r_\iota}$.}

By inverting every element of $S$, if necessary, we may assume the following.
\begin{equation}\label{eq:path}
\text{There exists some  segment of $\axis(r)$ of the form $e, p, re$.}
\end{equation}

Consider the case where $F\{T_\iota\} \ne F\{T_\tau\}$.  By~\eqref{eq:path}, we have a path
$\overline r p$ in $\axis(r)$ from \mbox{$\overline r \tau e \in \overline r T_\tau \ne T_\iota$} to $\iota e \in T_\iota$.
Now $\overline r p$ necessarily enters $T_\iota$ through an oriented edge of the form $g  e^{-1} $ for some
$g \in F_\iota$, and $g$ then lies in
 $\glue(F_\iota, \Eaxis(r)) -\gen{r}$. This proves the following.
\begin{equation}\label{eq:disconnected2}
\text{If   $\glue(F_\iota, \Eaxis(r)) \subseteq \gen{r}$  then
$F\{T_\iota\} = F\{T_\tau\}$.}
\end{equation}

Consider the case where $S_\iota$ is empty.
Here, $F_\iota\!=\!F_{\iota}/\normgen{S_{\iota}}\!=\!\{1\}$.
By~\eqref{eq:disconnected2},    $F\{T_\iota\} = F\{T_\tau\}$, and, then,
by~\eqref{eq:preconn}, there exists some $t\in F$ such that $F = \gp{t}{\quad}$.
Now  \mbox{$\{t, \overline t\} = \sqrt[F]{F} -\{1\} \supseteq S = \{r\}$.}
Hence, $F = \gp{r}{\quad}$, and, hence, $\glue(F, \Eaxis(r)) = \gen{r}.$
 This proves the following.
\begin{equation}\label{eq:empty}
\text{If  $\glue(F, \Eaxis(r)) \ne \gen{r}$ then $S_\iota$ is non-empty.}
\end{equation}

\medskip

\noindent\textbf{Case 2.1.}   $\glue(F, \Eaxis(r)) = \gen{r}$.

Here, by~\eqref{eq:disconnected2},   $F\{T_\iota\} = F\{T_\tau\}$.
Hence, $F(VT_\iota) = VT$ and $S = \lsup{F}{\mkern-2mu S}_\iota  \cup  \lsup{F}{\mkern-2mu r}$.
Consider any $v \in VT$.  We wish to show that $F_v = 1$,
and, we may assume that $v \in VT_\iota$.
Here, $F_v \le F_{\iota}$, and, since $F_\iota$ acts freely on $T_\iota$,  $F_v = 1$, as desired.
Thus, $F$ acts freely on $T$.
In~\eqref{eq:path}, the path $p$ from $\tau e$ to $r\iota e$ in $\axis(r)$ does not
meet $Fe$ since $\glue(F, \Eaxis(r)) = \gen{r}$, and,
hence, $p$ stays within~$T_\tau$, and, hence $r\iota e \in T_\tau$.
Thus, $r T_\iota = T_\tau$, and, by~\eqref{eq:preconn},   $F = F_\iota {\ast} \gp{r}{\quad}$.
Since  $F_{\iota}/\normgen{S_{\iota}}$ is trivial, we see that
$F/\normgen{S}$ is trivial, and, hence, $F/\normgen{S}$ is indicable.
Here all the required conclusions hold.

\medskip

\noindent \textbf{Case 2.2.} $\glue(F, \Eaxis(r)) \ne \gen{r}$.

By~\eqref{eq:empty},  $S_\iota$ is non-empty, and, hence,  there exists some
$e_\iota \in \bigcup\limits_{r_\iota  \in S_\iota} \Eaxis(r_\iota)$
such that\vspace{-2mm}  $$Fe_\iota
= \min\{Fe \mid e \in \bigcup\limits_{r_\iota  \in S_\iota} \Eaxis(r_\iota)\}.\vspace{-1mm}$$
 There then exists some
 $r_\iota \in S_\iota$ such that $e_\iota \in \Eaxis(r_\iota)$, and we
then know that  $\glue(F, \Eaxis(r_\iota))\!=\!\gen{r_\iota}$.
By~\eqref{eq:max}, $(r,e)\!=\!(r_\maxx,e_\maxx)$.
Using the definition of $(T,{<})$-stag\-gered modulo $F$, one can show that
$Fe_\iota$ does not meet the axis of any element of
$\lsup{F}{\mkern-2mu r}$, and, similarly,
$Fe_\iota$ does not meet the axis of any element of $\lsup{F}{\mkern-2mu S_\iota} - \lsup{F}{\mkern-2mu r}_\iota$.
It is clear that if $\lsup{F}{\mkern-2mu S_\tau} \ne \lsup{F}{\mkern-2mu S_\iota}$, then
$Fe_\iota$ does not meet the axis of any element of
 $\lsup{F}{\mkern-2mu S_\tau}$.
Hence, $Fe_\iota$ does not meet the axis of any element of
$S - \lsup{F}{\mkern-2mu r}_\iota$.
We now replace $(r,e)$ with $(r_\iota, e_\iota)$ in~\eqref{eq:disj},~\eqref{eq:max},
and we  then find that the same argument as before now terminates in Case 1 or Case 2.1.  Hence, all the desired
 conclusions hold.

\medskip

This completes the proof.
\end{proof}

\begin{Cor}\label{Cor:1} In
{\normalfont \mbox{Setting \ref{Sett:1}}},
suppose that $F$ is locally indicable  and
that $S$ is $T$-staggerable modulo $F$  and that \mbox{$\sqrt[F]{S}=S$.}
Then   $S \cup \Phi \subseteq F_{\nu}$
and
$F_{\nu}$ acts freely on~$T$  and    $\lsup{ F_{\nu} }{\mkern-2mu S}$ is a Whitehead subset of $F_{\nu}$.

\end{Cor}

\begin{proof}
 Recall from Definitions~\ref{defs:tree}
that $S$ has unique roots in $F$.

By Proposition~\ref{Prop:1}(i), $F_{\nu}$ is finitely generated and $S\cup\Phi \subseteq F_\nu$  and
$F_{\nu}/\normgen{S}$ has no infinite, cyclic quotients.
By Proposition~\ref{Prop:1}(ii),
$S$ is $T$-staggerable modulo $F_{\nu}$.
By Lemma~\ref{lem:bottom}(i),
$F_{\nu}/\normgen{S}$ is indicable, and, hence, trivial.  Now the result holds by Lemma~\ref{lem:bottom}(ii).
\end{proof}

 The next result is repeated from Corollary~A.2.3 of~\cite{ADL};
Section~A.3 of \mbox{\cite{ADL}} describes how these results continue
work of Magnus, Brodski\u{\i}, Howie, Short and others.

\begin{Thm}\label{Thm:1} Let $F$ be a locally indicable group  and let $T$ be an
$F$-tree with trivial edge stabilizers and
let $R$ be a subset of $F$  that
  is $T$-staggerable modulo $F.$

Then the following hold.
\begin{enumerate}[{\normalfont (i).}]
\item {\normalfont (The Local-indicability Freiheitssatz)}  $\gen{\lsup{F}{\mkern-2mu R}}$ acts freely on $T$,
and, hence, $\gen{\lsup{F}{\mkern-2mu  R}}$ is a free group and the $F$-stabilizer of each vertex of $T$
embeds in $F/\normgen{R}$ under the natural map.
\item  $F/\normgen{\mkern-6mu\sqrt[F]{R}}$ is locally indicable.
\end{enumerate}
\end{Thm}

\begin{proof} Since $\gen{\lsup{F}{\mkern-2mu R}} \le \gen{\lsup{F}{\mkern-2mu (\sqrt[F]{R})}}$, we may
replace $R$ with $\sqrt[F]{R}$ and then we have
$\sqrt[F]{R}=R$.

\noindent\mbox{\phantom{...ii}(i).}   Let $S$  be an arbitrary finite subset  of $\lsup{F}{\mkern-2mu R}$ and
let $\Phi$   be the empty set.  We  may then assume that we are in Setting~\ref{Sett:1}.
 By Corollary~\ref{Cor:1}, $\gen{S} \le F_{\nu}$  and   $F_{\nu}$ acts freely on~$T$.
Hence,
$\gen{S}$ acts freely on~$T$, and, since $S$ is an arbitrary finite subset  of $\lsup{F}{\mkern-2mu R}$,
we see that $\gen{\lsup{F}{\mkern-2mu R}}$ acts freely on $T$, that is,
the $F$-stabilizer of each vertex of $T$
embeds in $F/\normgen{R}$ under the natural map.  By Reidemeister's Theorem, or by Bass-Serre Theory,
$\gen{\lsup{F}{\mkern-2mu R}}$ is a free group.

\noindent\mbox{\phantom{...i}(ii).}   Let $H$ be an arbitrary
finitely generable subgroup of $F/\gen{\lsup{F}{\mkern-2mu R}}$.
It suffices to show that $H$ is indicable.
We may assume that $H$
has no infinite, cyclic quotients.  Then $H^\text{ab}$ is finite.
Let $d$ denote the exponent of $H^\text{ab}$.
There exists some finite subset $\Phi$ of $F$ such that
\mbox{$H=(\gen{\Phi}\,\gen{\lsup{F}{\mkern-2mu R}})/\gen{\lsup{F}{\mkern-2mu R}}$}.
Now \mbox{$H'=(\gen{\Phi}'\,\gen{\lsup{F}{\mkern-2mu R}})/\gen{\lsup{F}{\mkern-2mu R}}$}
and we see that
 $\Phi^d  \subseteq \gen{\Phi}'\,\gen{\lsup{F}{\mkern-2mu R}}$.  Hence,
there exists some finite subset
 $S$ of $\lsup{F}{\mkern-2mu R}$ such that  the finite set $\Phi^d$ lies in $\gen{\Phi}'\gen{S}$.
Then, the abelian group \mbox{$(\gen{\Phi \cup S}/\normgen{S})^{\text{ab}}$} has exponent at most~$d$,
and, hence, \mbox{$\gen{\Phi \cup S}/\normgen{S}$} has no infinite, cyclic quotients.
We  may then assume that we are in Setting~\ref{Sett:1}. By Corollary~\ref{Cor:1},
\mbox{$\Phi \subseteq  F_\nu   = \gen{\lsup{ F_{\nu} }{\mkern-2mu S}}
\le  \gen{\lsup{F}{\mkern-2mu R}}$,} and, hence, $H$ is trivial, and hence, $H$ is indicable.
\end{proof}

\begin{Cor}[Brodski\u{\i}-Howie-Short]\label{Cor:Howie}
Suppose that $A$ and $B$ are locally indicable groups and that
 $r$ is an element of $A{\ast}B$.  If no $A{\ast}B$-conjugate of $r$ lies in $A$, then
  the natural map  embeds $A $   in $(A{\ast}B)/\normgen{r}$. \hfill\qed
\end{Cor}

\begin{Cor}[Magnus' Freiheitssatz]
Suppose that $A$ and $B$ are free groups and that
 $r$ is an element of $A{\ast}B$.  If no $A{\ast}B$-conjugate of $r$ lies in $A$, then
  the natural map  embeds $A $   in $(A{\ast}B)/\normgen{r}$. \hfill\qed
\end{Cor}

\section{Moving the new information back up the chain}\label{sec:5}

In this section, we recover the Local-indicability Cohen-Lyndon Theorem.

\begin{Lem}  Let $Y$ be a non-empty subset of $F$.

Then $Y$ is a Whitehead  subset of $F$ if and only if all of the following hold:
$1 \not\in Y$; for each $y \in Y$, $\mathbf{C}_F(y) = \gen{y}$;  $Y$ generates $F$;
$Y$ is closed under the $F$-conjugation action; and, $Y$ with the $F$-conjugation
action  is the vertex set of an $F$-tree with trivial edge stabilizers.
\end{Lem}

\begin{proof}  Suppose that $Y$ is  a Whitehead subset of  $F$, and let $X$ be
a free-generating set of $F$ such that $Y = \lsup{F}{X}$.
Then $1 \not\in Y$,  $Y$ generates $F$, $Y$ is closed under the $F$-conjugation action,
 and each $y\in Y$ generates a non-trivial free factor of $F$ and, hence, $\mathbf{C}_F(y) = \gen{y}$.
We can  express $F$ as the fundamental group of a tree of groups
in which the edge groups are trivial and the family of vertex groups is
$(\gen{x} \mid  x \in X )$.  Let   $T$ denote the corresponding
Bass-Serre tree.
Thus
$VT  = \mathop{\textstyle\bigvee}\limits_{x\in X} F/ \gen{x}
\simeq \mathop{\textstyle\bigvee}\limits_{x\in X} \lsup{F}{\mkern-2mu x} = \lsup{F}{\mkern-2mu X} =Y,$\vspace{-2mm}
and $ET$ is some free $F$-set.

Conversely, suppose that
$1 \not\in Y$ and that $Y$ generates $F$ and that, for each $y \in Y$, $\mathbf{C}_F(y) = \gen{y}$  and
that $Y$ is closed under the $F$-conjugation action and that $Y$ with the $F$-conjugation
action is the vertex set of an $F$-tree $T$ with trivial edge stabilizers.
By the Bass-Serre Structure Theorem,
there exists some graph of groups
$(\mathcal{F}, \overline T)$
and some maximal subtree $ \overline T_0$ of~$ \overline T$  such that
$F$ is the fundamental group of $(\mathcal{F},  \overline T, \overline T_0)$.
By the centralizer condition, the family of vertex groups can be expressed as
$(\gen{x} \mid x \in X)$ for some transversal~$X$  for the  $F$-conjugation action on $ Y.$
Since $F$  is then generated by the
$F$-conjugates of the vertex groups,
it follows that $ \overline T =  \overline T_{0}$.  Hence,  $F$
is the free product of the family of vertex groups.  Since $1 \not\in X$, $X$ is a
free-generating set of $F$ and $Y = \lsup{F}{\mkern-2mu X}$, as desired.
\end{proof}

\begin{Cor}\label{Cor:vertex}  Let $R$ be a non-empty subset of $F$.

 Then $R$ is   Cohen-Lyndon aspherical in  $F$  if and only if all of the following hold:
 $\,1 \not\in R$;  no two distinct elements of $R$ are $F$-conjugate;   for each $r \in R$,
\mbox{$\mathbf{C}_{\gen{\lsup{F}{\mkern -2mu R}}}(r) = \gen{r}$;} and,
$\lsup{F}{\mkern -2mu R}$ with the $\gen{\lsup{F}{\mkern -2mu R}}$-conjugation action
 is the vertex set of an $\gen{\lsup{F}{\mkern -2mu R}}$-tree with trivial edge
stabilizers. \hfill\qed
\end{Cor}

We now give a  generalization of   results of Cohen and Lyndon~\cite[Lemma~2.1]{CohenLyndon}
and Karrass and Solitar \cite[Theorem~1]{KarrassSolitar} that is similar to
a result of Chiswell, Collins and Huebschmann~\cite[Theorem~4.6]{CCH}.

\begin{Thm}\label{Thm:KS1} Let
$(\mathcal{F}, Y)$  be   a graph of groups  with family of groups $(\mathcal{F}(y) \mid y \in Y)$
and family of edge maps $(\overline t_e \colon   \mathcal{F}(e) \to \mathcal{F}(\tau_Y e) \mid e \in EY)$,
let $Y_0$ be a maximal subtree  of~$Y$, and let $F$ be the  fundamental group of
$(\mathcal{F}, Y,Y_0)$.

Let $R \coloneq  (r_v \mid v \in VY)$ be a family of elements of $F$ with the property that,
for each $v \in VY,$  all of the following hold: $r_v \in \mathcal{F}(v)$;  for each $e \in \iota_Y^{-1} (\{v\})$,
\mbox{$\gen{\lsup{\mathcal{F}(v)}{r_v}} \cap \,\,\mathcal{F}(e) = \{1\}$}; and,  for each $e \in \tau_Y^{-1} (\{v\})$,
\mbox{$\gen{\lsup{\mathcal{F}(v)}{r_v}} \cap \,\,\lsup{\overline t_e}{\mathcal{F}(e)} = \{1\}$.}

  Then, for each $v \in VY$,
 $\gen{\lsup{F}{R}} \cap \,\, \mathcal{F}(v) = \gen{\lsup{\mathcal{F}(v)}{r_v}}$.

Let $U \coloneq \{v \in VY \mid r_v \ne 1\}$.  Then $\{r_u \mid u \in U\}$ is Cohen-Lyndon aspherical in $F$  if and only if, for each $u \in U$, $r_u$ is
Cohen-Lyndon aspherical in $\mathcal{F}(u)$.
\end{Thm}

\begin{proof}
For each $v \in VY$, let $\overline{\mathcal{F}}(v) \coloneq \mathcal{F}(v)/\normgen{r_v}$, and,
for each $e \in EY$, let $\overline{\mathcal{F}}(e) \coloneq \mathcal{F}(e)$.
By the hypotheses on  $R$, the edge maps for $\mathcal{F}$ then induce injective maps which
give a graph of groups
$(\overline{\mathcal{F}}, Y)$  with family of groups $( \overline{\mathcal{F}}(y) \mid y \in Y)$.
Let $\overline F$ denote the  fundamental group of
$(\overline{\mathcal{F}}, Y,Y_0)$.  Using the definition of fundamental groups, we then
 construct a natural group homomorphism
$F \mapsto \overline F$, which is
surjective and has kernel   $\gen{\lsup{F}{\mkern-1.6mu R}}$.
In particular, for each $v \in VY$,
\mbox{$\gen{\lsup{F}{\mkern-1.6mu R}} \cap \mathcal{F}(v) = \gen{\lsup{\mathcal{F}(v)}{r_v}}$.}

Let $T$ denote the Bass-Serre tree for $(\mathcal{F}, Y,Y_0)$, and let
$\overline T$ denote the Bass-Serre tree for $(\overline{\mathcal{F}}, Y,Y_0)$.
Using the definition  of Bass-Serre trees, we construct
a natural identification    $\gen{\lsup{F}{\mkern-1.6mu R}}\leftmod T  = \overline T$.
Using the Bass-Serre Structure Theorem for $\gen{\lsup{F}{\mkern-1.6mu R}}$ acting on $T$, we see that
$\gen{\lsup{F}{\mkern-1.6mu R}}$ is the fundamental group of a tree of groups over~$\overline T$, with trivial
edge groups, and with family of vertex groups of the form
\mbox{$(\gen{\lsup{F}{\mkern-2mu R}} \cap\,\, \lsup{z}{\mathcal{F}(v)} \mid v \in VY, z \in Z_v)$,} where,
 for each $v \in VY$,  $Z_v$ is some transversal for the
$\gen{\lsup{F}{\mkern-2mu R}}\mathcal{F}(v)$-action on~$F$ by multiplication on the right.
Thus \mbox{$\gen{\lsup{F}{\mkern-2mu R}}
= \mathop{\ast}\limits_{v \in VY} \mathop{\ast}\limits_{z \in Z_v}  \gen{\lsup{z\mathcal{F}(v)}{r_v}}
= \mathop{\ast}\limits_{u \in U} \mathop{\ast}\limits_{z \in Z_u}  \gen{\lsup{z\mathcal{F}(u)}{r_u}}$.}

We claim that, for each $u \in U$, $\mathbf{C}_{F}(r_u) = \mathbf{C}_{\mathcal{F}(u)}(r_u)$.
If $f \in \mathbf{C}_{F}(r_u)$, then $r_u$ fixes the path in $T$ from $1 \mathcal{F}(u)$ to
$f \mathcal{F}(u)$, and, by the hypotheses, $r_u$ fixes no edges of $T$. Hence the path under consideration
is trivial, and, hence,
$f \in \mathcal{F}(u)$, as claimed.

Suppose that, for each $u \in U$, $r_u$ is
Cohen-Lyndon aspherical in  $\mathcal{F}(u)$, and let $Y_u$ be a transversal for the
$\gen{\lsup{\mathcal{F}(u)}{r_u}}\mathbf{C}_{\mathcal{F}(u)}(r_u)$-action on  $\mathcal{F}(u)$
by multiplication on the right such that $\lsup{Y_u}{r_u}$ is a free-generating set of
$\gen{\lsup{\mathcal{F}(u)}{r_u}}$.  Then $Z_uY_u$ is a transversal for
the $\gen{\lsup{F}{\mkern-2mu R}}\mathbf{C}_{F}(r_u)$-action  on  $F$ by multiplication on the right
and $\bigvee\limits_{u\in U} \lsup{Z_uY_u}{r_u}$ is a free-generating set of $\gen{\lsup{F}{\mkern-2mu R}}$.
Thus
$R$ is Cohen-Lyndon aspherical in $F$.

Conversely, suppose that $R$ is  Cohen-Lyndon aspherical in $F$  and that $u$ is an element of~$U$.
By Corollary~\ref{Cor:vertex},
$1 \not\in R$, no two distinct elements of $R$ are $F$-conjugate, and
there exists some $\gen{\lsup{F}{\mkern-2mu R}}$-tree $T'$ with vertex set
$ \lsup{F}{\mkern-2mu R} $  and with trivial edge stabilizers.
By Corollary~\ref{Cor:vertex}, it remains to show that
$\lsup{\mathcal{F}(u)}{r_u}$ is the vertex set of some
$\gen{\lsup{\mathcal{F}(u)}{r_u}}$-tree
with trivial edge stabilizers.
The subgroup $\gen{\lsup{\mathcal{F}(u)}{r_u}}$ of $\gen{\lsup{F}{\mkern-2mu R}}$
acts on $T'$ and for each  vertex of $T'$,  say $ \lsup{f}{r_{u'}} $
with $f \in F$
and   $u' \in U$, the  $\gen{\lsup{\mathcal{F}(u)}{r_u}}$-stabilizer is
 \mbox{$\gen{\lsup{\mathcal{F}(u)}{r_u}} \cap \lsup{f}{ (\mathbf{C}_{\mathcal{F}(u')}(r_{u'}))}$.}
The latter intersection fixes the
path in $T$ from $1\mathcal{F}(u)$ to $f \mathcal{F}(u')$, and, hence, is
 trivial unless $u'=u$ and $f \in \mathcal{F}(u)$.  Thus the
set of vertices of $T'$ with non-trivial
$\gen{\lsup{\mathcal{F}(u)}{r_u}}$-stabilizer is the subset
$\lsup{\mathcal{F}(u)}{r_u}$ of $\lsup{F}{R} \,\,\,(= VT')$.
Since the vertex stabilizers then generate $\gen{\lsup{\mathcal{F}(u)}{r_u}}$
we see that $\gen{\lsup{\mathcal{F}(u)}{r_u}}\leftmod T'$ is a tree.
Hence, successively $\gen{\lsup{\mathcal{F}(u)}{r_u}}$-equivariantly
contracting suitable edges of $T'$ produces an $\gen{\lsup{\mathcal{F}(u)}{r_u}}$-tree
with vertex set $\lsup{\mathcal{F}(u)}{r_u}$ and trivial edge stabilizers, as desired.
\end{proof}

\begin{Rem}\label{Rem:KS1} If $A$ and $B$ are groups and $r$ is an element of $A$,
it follows from Theorem~\ref{Thm:KS1}
that $r$ is Cohen-Lyndon aspherical in $A$  if and only if $r$ is Cohen-Lyndon aspherical in $ A{\ast}B $.
\hfill\qed
\end{Rem}

\begin{Cor}\label{Cor:KS} Suppose that $A$ and $B$ are groups  and that  $A_1$ is a subgroup of $A$
and that  $B_1$ is a subgroup of $B$ and that  $r$ is an element of $A_1{\ast}B_1$ and that
  the natural maps embed $A_1$ and $B_1$  in $(A_1{\ast}B_1)/\normgen{r}$.

 If $r$ is  Cohen-Lyndon aspherical in $A_1{\ast}B_1$,  then $r$ is   Cohen-Lyndon aspherical in $A{\ast}B$.
\end{Cor}

\begin{proof}
 By applying Theorem~\ref{Thm:KS1} to  \mbox{$(A_1 {\ast} B_1){\ast}_{B_1}B\,\,\,\,\,(=A_1 {\ast} B)$,} we see that $r$ is
Cohen\d1Lyndon aspherical in  $A_1 {\ast} B$ and also  the natural map  embeds
  $(A_1 {\ast} B_1)/\normgen{r}$ in $(A_1 {\ast} B)/\normgen{r}$.
Hence, the natural map  embeds   $A_1$ in $(A_1 {\ast} B)/\normgen{r}$.

 By applying Theorem~\ref{Thm:KS1}  to  $(A_1 {\ast} B){\ast}_{A_1}A \,\,\,\,\,(=A {\ast} B)$, we see that
 $r$ is  Cohen-Lyndon aspherical in $A {\ast} B$.
\end{proof}

\begin{Cor}\label{Cor:KS4}  Suppose that $F$ is  locally indicable  and that
 $T$ is an $F$-tree with trivial edge stabilizers and that $r$ is
an element of $F$ which fixes no vertex of $T$.

 Let
 $H$ be a  subgroup of~$F$ such that  $H \supseteq \glue(F, \Eaxis(r))$ and
$H = \gen{\glue(H, \axis(r))}$.
  If   $r$ is  Cohen-Lyndon aspherical in $H$, then $r$ is Cohen-Lyndon aspherical in~$F$.

\end{Cor}

\begin{proof}   Let $e$ be an edge in $\axis(r)$.   Let $\overline T$ denote the $F$-tree obtained from
$T$ by collapsing all edges in $ET-Fe$.  Then $E\overline T = Fe$ and  $r$ shifts $e$ and
  $\Eaxis_{\overline T}(r) = Fe \cap \Eaxis_T(r)$ and
$\glue(H, \axis_{ T}(r))  \subseteq \glue(H, \axis_{\overline T}(r))$.
It follows that we can replace $T$ with $\overline T$ and assume that $ET = Fe$.

Let $T' = H(\axis(r))$.   By Remark~\ref{Rem:tree}, $T'$ is an $H$-subtree of $T$.

For each $e' \in \Eaxis(r)$, there exists some $f \in F$ such that $fe = e'$,
and then \mbox{$f \in \glue(F,\Eaxis(r)) \subseteq H$} and $e' = fe \in He$.  Hence $\Eaxis(r) \subseteq He$.
Hence $ET' = H(\Eaxis(r)) \subseteq He$.  In summary, $ET = Fe$ and $ET'=He$.

\medskip

\noindent \textbf{Case 1.}  $\abs{F \leftmod  (V T )}=2$.

Here, we can write $F  = A{\ast}B$
and no $F$-conjugate of $r$  lies in   $A{\cup}B$,  and
 $H = A_1{\ast}B_1$ for some $A_1 \le A$ and some $B_1 \le B$.
By Corollary~\ref{Cor:Howie} and Corollary~\ref{Cor:KS},
$r$ is  Cohen-Lyndon aspherical in $ A{\ast}B\,\,\,( =F) $.

\medskip

\noindent \textbf{Case 2.}  $\abs{F  \leftmod  (V T )}=1$.

Here, we can write $F   = A \ast (b\colon \{1\}\to\{1\})$ and no $F$-conjugate of $r$  lies in   $A$.

\medskip

\noindent \textbf{Case 2.1.}  $\vert{H \leftmod  (V T')}\vert=2$.

Here,   $H = A_1{\ast}\lsup{\,b}{\mkern-2mu A_2}$ for some $A_1 \le A$ and some $A_2 \le A$.

Then $r$ is Cohen-Lyndon aspherical in  $ A{\ast}\lsup{\,b}{\mkern-2mu A} $, by
Corollary~\ref{Cor:Howie} and Corollary~\ref{Cor:KS}.

 Now
$r$ is Cohen-Lyndon aspherical in
$(A{\ast}\lsup{\,b}{\mkern-2mu A})  \ast (b \colon A \to \,\lsup{b}{\mkern-2mu A})\,\,\,\,\,\,(= A {\ast} \gp{b}{\,\,\,} = F)$, by
Corollary~\ref{Cor:Howie} and Theorem~\ref{Thm:KS1}.

\medskip

\noindent \textbf{Case 2.2.}   $\vert{H \leftmod  (V T')}\vert=1$.

Here,   $H = A_1{\ast}(ba\colon \{1\}\to\{1\}) $ for some $A_1 \le A$ and some $a \in A$.

If no $F$-conjugate of $r$ lies in
$\gp{ba}{\,\,\,}$, then $r$ is Cohen-Lyndon aspherical in
$ A{\ast}\gp{ba}{\,\,\,}\,\,\,\,\,(= A{\ast} \gp{b}{\,\,\,}= F)$, by
Corollary~\ref{Cor:Howie} and Corollary~\ref{Cor:KS}.

Thus we may assume that $r$  itself lies in $\gp{ba}{\,\,\,}$ and $r$ is then clearly
Cohen-Lyndon aspherical in
$\gp{ba}{\,\,\,}$. Hence, $r$ is Cohen-Lyndon aspherical in  $A{\ast}\gp{ba}{\,\,\,} \,\,\,\,\,(= A{\ast} \gp{b}{\,\,\,}= F)$, by
Remark~\ref{Rem:KS1}.

\medskip

Hence, in all cases,  $r$ is  Cohen-Lyndon aspherical in $F$.
\end{proof}

\begin{Not}  If $a$ and $b$ are elements of $F$, we shall let
 $ \lsup{\gp{a}{\quad}}{b} \coloneq (\lsup{a^i}{b} \mid i \in \Z)$.
For any subset $J$ of $\Z$, we let \mbox{$\lsup{a^{J}}{b}
\coloneq  (\lsup{a^j}{b}  \mid j \in J) \subseteq   \lsup{\gp{a}{\quad}}{b}$.}
\hfill\qed
\end{Not}

The following will be useful for simplifying calculations.

\begin{Lem}\label{Lem:add gen}  Suppose that $f$ is an element of $F$ and that
$N$ is a normal subgroup of $F$ such that \mbox{$F/N = \gp{fN}{\quad}$.}
Let $z$ be a symbol, let \mbox{$\tilde F \coloneq F {\ast}\gp{z}{\quad}$}, and let $\tilde N$ denote the
smallest normal subgroup of $\tilde F$ containing $N \cup \{f\overline z\}$.
Then \mbox{$\tilde F/\tilde N = \gp{f\tilde N}{\quad} =  \gp{z\tilde N}{\quad}$} and $N$ is a free factor of $\tilde N$.
\end{Lem}

\begin{proof} It is easy to see that $ F   =  \gp{f}{\quad} \ltimes  N  $ and
  \mbox{$\tilde F/\tilde N = \gp{f\tilde N}{\quad} = \gp{z\tilde N}{\quad}$.}\vspace{.2mm}
It is not difficult to use generators and relations to check that
 \mbox{$\tilde F = F {\ast}\gp{z}{\quad} =
 \gp{f}{\quad} \ltimes (N{\ast}\gp{\lsup{\gp{f}{\quad}}{(f\overline z)}}{\quad}),$}
and, hence, $\tilde N = N{\ast}\gp{\lsup{\gp{f}{\quad}}{(f\overline z)}}{\quad}$.
\end{proof}

The following is proved  but not stated  in~\cite{EdjvetHowie};   here, we  give an argument
with a different rewriting procedure,  to provide some variety.

\begin{Lem}{\normalfont(Edjvet-Howie)}\label{Lem:EH}
Suppose that $A$ and $B$ are locally indicable groups,
and that $r$ is an element of $A{\ast}B$ such that no $A{\ast}B$-conjugate of $r$ lies in $A{\cup} B$.

Let  $N$ be a normal subgroup of $A{\ast}B$ such that $r \in N$ and $(A{\ast}B)/N$ is infinite and cyclic.
 If $r$ is  Cohen-Lyndon aspherical in $N$, then $r$ is Cohen-Lyndon aspherical in~$A{\ast}B$.
\end{Lem}

\begin{proof} Let $F = A{\ast}B$.

It follows  from
Lemma~\ref{Lem:add gen} and  Remark~\ref{Rem:KS1}   that,
by adjoining an infinite, cyclic free factor to $A$  if necessary,
we may assume that there exists some $a \in A$   such that $F/N = \gp{aN}{\quad}$.

Similarly, by adjoining an infinite, cyclic free factor  to $B$  if necessary,
we may further assume that there exists some $b \in   B$   such that $F/N  = \gp{aN}{\quad} =   \gp{bN}{\quad}$.

Let $N_A \coloneq A \cap N$ and let $N_B \coloneq B \cap N$.
Notice that $\lsup{a}{N_A} = N_A$ and   $\lsup{a}{N_B} = \lsup{a \overline b}{N_B}$.
It is not difficult to use generators and relations to check that
 $F   =  \gp{a}{\quad} \ltimes (N_A{\ast}N_B{\ast}\gp{\lsup{\gp{a}{\quad}}{(a\overline b)}}{\quad}), $
and, hence,  \mbox{$N = N_A{\ast}N_B{\ast}\gp{\lsup{\gp{a}{\quad}}{(a\overline b)}}{\quad}$.}

We may replace $r$ with any $F$-conjugate of $r$, since any automorphism  of $N$ respects
Cohen-Lyndon asphericity.

Let $j \in \Z$.  Conjugation by $a^j$ induces an automorphism of $N$ and we obtain
the free-product decomposition
\mbox{$N = N_A \ast \lsup{a^j}{\mkern-2mu N_B} \ast \gp{\lsup{\gp{a}{\quad}}{(a\overline b)}}{\quad}$};
here, we consider the   resulting cyclically reduced
expression  for $r$, and  we see that
there exists some finite subset $J_j$ in $\Z$ such that this cyclically reduced
expression   lies  in
$N_A{\ast}\,\,\lsup{a^j}{\mkern-2mu N_B}{\ast}\,\,\gp{\lsup{a^{J_j}}{\mkern-2mu (a\overline b)}}{\quad}$.
 We may assume that $J_j$ is minimal and
 that $r$ itself lies in this free factor of $N$.
If $J_j$ is empty, then,
by   replacing $r$ with $\lsup{\overline a^j}{r}$, we may assume that $r \in N_A{\ast}N_B$, and then,
by Remark~\ref{Rem:KS1},  $r$~is
Cohen-Lyndon aspherical in  \mbox{$N_A{\ast}  N_B $}.
By  Corollary~\ref{Cor:Howie} and  Corollary~\ref{Cor:KS},
$r$~is  Cohen-Lyndon aspherical in $A{\ast}B \,\,\,\,(= F)$.

Thus we may assume that, for each $j\in \Z$, $J_j$ is non-empty.

Let $j \in \Z$.  Then $\lsup{a^{j{+}1}}{\mkern-2mu  N_B}
 = \lsup{(\lsup{a^j}{\mkern-2mu (a\overline b)})}{(\lsup{a^j}{\mkern-2mu N_B})}$,
and we find that $J_{j}  \subseteq J_{j{+}1} \cup \{j\}$.
Similarly,  \mbox{$J_{j{+}1}  \subseteq J_j \cup \{j\}$}, and, hence,
$J_{j} \cup \{j\} = J_{j{+}1} \cup \{j\}$.

It follows that \mbox{$\{j\in\Z \mid   \min J_{j} < j\}
=\{j\in\Z \mid \min J_{j{+}1} = \min J_{j} < j\} = \{j\in\Z \mid \min J_{j{+}1} < j\}$.}  It
is not difficult to see that this set  cannot be all of $\Z$.
Let $K \coloneq  \{j\in\Z \mid   \min J_{j} \ge j\}$. Then
\mbox{$K=  \{j\in\Z \mid \min J_{j{+}1} \ge j\}$} and $K \ne \emptyset$.
Let us choose $k\in K $ to minimize
the pair $(\,\abs{J_k}, (\min J_k) - k)$.
If  $k{+}1 \in K$, then $\min J_{k{+}1} \ge k{+}1$,
and  then  $k \not\in J_{k{+}1}$, and
then  $J_{k{+}1}\subseteq  J_k$, and then, by the minimality property of $k$,
  $J_{k{+}1}=J_k$  and
\mbox{$(\min J_{k{+}1})-(k{+}1) \ge (\min J_k) -k$,} which is a contradiction.
Thus  $k{+}1 \in \Z -K$.  It follows that  $\min J_{k{+}1} = k$.

By replacing $r$ with $\lsup{\overline a^k}{\mkern-2mu r}$, we may assume that $k=0$, and, hence, $\min J_1 =0$.
Let $\nu \coloneq \max J_1$.
Then $\nu \ge 0$ and
 $$r\in N_A{\ast}\,\,\lsup{a}{N_B}{\ast}\,\,\gp{\lsup{a^{[0{\uparrow}\nu]}}{(a\overline b)}}{\quad}
=N_A{\ast}\,\, \lsup{a\overline b}{N_B}{\ast}\,\,\gp{\lsup{a^{[0{\uparrow}\nu]}}{(a\overline b)}}{\quad}
 =N_A{\ast}\,\,  N_B{\ast}\,\, \gp{\lsup{a^{[0{\uparrow}\nu]}}{(a\overline b)}}{\quad}.$$

Since  $\min J_1 = 0$, no $N$-conjugate of
$r$ lies in
 $N_A{\ast}\,\,\lsup{a}N_B{\ast}\,\, \gp{\lsup{a^{[1{\uparrow} \nu ]}}{(a\overline b)}}{\quad}
=N_A{\ast}\,\,\lsup{a\overline b}N_B{\ast}\,\, \gp{\lsup{a^{[1{\uparrow} \nu ]}}{(a\overline b)}}{\quad}.$

We claim that no $N$-conjugate of
$r$ lies in $N_A{\ast}N_B{\ast}\,\, \gp{\lsup{a^{[0{\uparrow}(\nu{-}1)]}}{(a\overline b)}}{\quad}$.
The case of the claim where $\nu = 0$ holds because $J_0$ is non-empty.
The case of the claim where $\nu \ge 1$ holds because $\max J_1 =  \nu $ and, hence,
no $N$-conjugate of
$r$ lies in
$$N_A{\ast}\,\,\lsup{a}N_B {\ast}\,\, \gp{\lsup{a^{[0{\uparrow}(\nu{-}1)]}}{(a\overline b)}}{\quad}
= N_A{\ast}\,\,\lsup{a\overline b}N_B{\ast}\,\, \gp{\lsup{a^{[0{\uparrow}(\nu{-}1)]}}{(a\overline b)}}{\quad}
= N_A{\ast}N_B{\ast}\,\, \gp{\lsup{a^{[0{\uparrow}(\nu{-}1)]}}{(a\overline b)}}{\quad}.$$
This proves the claim.

By Remark~\ref{Rem:KS1},  $r$~is
Cohen-Lyndon aspherical in  \mbox{$N_A{\ast}  N_B
{\ast}\,\, \gp{\lsup{a^{[0{\uparrow}\nu]}}{\mkern-2mu (a\overline b)}}{\quad}$}.
By Corollary~\ref{Cor:Howie} and Theorem~\ref{Thm:KS1}, $r$~is  Cohen-Lyndon aspherical in
the HNN extension
$$  (N_A{\ast}N_B{\ast}\,\, \gp{\lsup{a^{[0{\uparrow}\nu]}}{(a\overline b)}}{\quad})
 \ast
(a \colon(N_A{\ast}N_B{\ast}\,\,\gp{\lsup{a^{[0{\uparrow}(\nu{-}1)]}}{(a\overline b)}}{\quad})\to
(N_A{\ast}\,\,\lsup{a\overline b}N_B{\ast}\,\, \gp{\lsup{a^{[1{\uparrow} \nu ]}}{(a\overline b)}}{\quad})),$$
which is~$ \gp{a}{\quad} \ltimes (N_A{\ast}N_B{\ast}\gp{\lsup{\gp{a}{\quad}}{(a\overline b)}}{\quad}) \,\,\,\,\,(= F)$.
\end{proof}

\begin{Cor}\label{Cor:EH}
Suppose that $F$ is  locally indicable  and that
 $T$ is an $F$-tree with trivial edge stabilizers and that $r$ is
an element of $F$ which fixes no vertex of $T$.

Let   $N$ be
a normal subgroup of $F$ such that $r\in N$ and $F/N$ is infinite and cyclic.
  If  $r$ is  Cohen-Lyndon aspherical in $N$, then $r$ is Cohen-Lyndon aspherical in~$F$.
\end{Cor}

\begin{proof}  The argument is similar to the proof of Corollary~\ref{Cor:KS4}.
Let $e$ be an edge in $\axis(r)$.  Let $\overline T$ denote the
$F$-tree obtained from $T$ by contracting all edges in $ET-Fe$.  Then $E\overline T = Fe$ and
$r$ shifts $e$.  Hence $r$ fixes no vertex of $\overline T$.  Thus we may replace $T$ with $\overline T$
and assume that $ET = Fe$.

\medskip

\noindent \textbf{Case 1.}  $\abs{F \leftmod  (V T )}=2$.

Here, we can write  $F = A{\ast}B$
and no $F $-conjugate of $r$  lies in   $A{\cup}B$.
Then, $r$ is Cohen-Lyndon aspherical in  $ F$  by Lemma~\ref{Lem:EH}.

\medskip

\noindent \textbf{Case 2.}  $\abs{F \leftmod  (V  T )}=1$.

Here, we can write $F  = A{\ast}(b \colon\{1\}\to\{1\})$ and  no $F $-conjugate of $r$  lies in $A$.

If no $F$-conjugate of $r$ lies in $\gp{b}{\quad}$, then
$r$ is Cohen-Lyndon aspherical in  $ A{\ast}\gp{b}{\quad} \,\,\,\,\,(= F) $ by Lemma~\ref{Lem:EH}.

 Thus we may assume that  $r$ itself lies in $\gp{b}{\quad}$  and
$r$ is then clearly Cohen-Lyndon aspherical in $\gp{b}{\quad}$.  Hence, $r$ is
Cohen-Lyndon aspherical in  $ A{\ast}\gp{b}{\quad} \,\,\,\,\,(= F )$,
by Remark~\ref{Rem:KS1}.

\medskip

 Thus, in all cases,  $r$ is  Cohen-Lyndon aspherical in  $ F $.
\end{proof}

\begin{Thm5}\label{thm:CL}   Suppose that $F$ is a locally indicable group
and that $T$ is an   $F$-tree  with trivial edge stabilizers.
If $r$ is  an element of $F$ that fixes no vertex of $T$,
 then $\mathbf{C}_F(r)$ is infinite and cyclic, and $r$   is  Cohen-Lyndon aspherical in $F$.
\end{Thm5}

\begin{proof}
Recall from Definitions~\ref{defs:tree} that $r$ has a unique root  in $F$
and that $\mathbf{C}_F(r) = \gen{\mkern-6mu \sqrt[F]{r}}$.

Let $S \coloneq \{\sqrt[F]{r}\,\}$ and $\Phi \coloneq \emptyset$.
We  may then assume that we are in Setting~\ref{Sett:1}, and  that
$v $ is a vertex in $\axis(\mkern-6mu \sqrt[F]{r})$.  Then $\gen{S}Y = \axis(\mkern-6mu \sqrt[F]{r}) = \axis(r)$.
 We shall use decreasing induction to show that, for each $n \in [0{\uparrow}\nu]$,
 $r$~is  Cohen-Lyndon aspherical in  $F_{n}$.

By Corollary~\ref{Cor:1}, $\lsup{F_{\nu} }{\mkern-2mu S}$ is a Whitehead subset of  $F_{\nu}$.  Since
$\abs{S} =1$,  $F_{\nu}$ is cyclic, and, hence, $\lsup{F_{\nu} }{\mkern-2mu S} = S$.
Thus, $F_{\nu} = \gp{\sqrt[F]{r}}{\quad}$.  It is then clear
that $r$ is Cohen-Lyndon aspherical in~$F_{\nu}$.

Now suppose that $n \in [0{\uparrow}(\nu{-}1)]$ and that $r$ is   Cohen-Lyndon aspherical in $ F_{n+1}$.

First, we wish to show that $r$ is Cohen-Lyndon aspherical in  $ F_n^\dag$.
This is trivial if \mbox{$F_n^\dag =F_{n+1}$,} and we
consider only the case where $F_n^\dag/F_{n+1}$ is infinite and cyclic.
Here,    $r$ is  Cohen-Lyndon aspherical in  $ F_n^\dag$ by Corollary~\ref{Cor:EH}.

We next want to show that  $r$ is Cohen-Lyndon aspherical in $F_n$.

Recall that $F_n^\dag$ is generated by
$\glue(F_n^\dag, Y) $ and that  \mbox{$\glue(F_n, EY) \subseteq F_n^\dag$.}
Now $$\glue(F_n^\dag, Y) \subseteq \glue(F_n^\dag, \axis(r))$$ and  $$\glue(F_n, \Eaxis(r)) =  \glue(F_n, \gen{S}(EY)) \subseteq \gen{S}\glue(F_n, EY)\gen{S} \subseteq F_n^\dag.$$
Here,  $r$ is  Cohen-Lyndon aspherical in $F_n$ by Corollary~\ref{Cor:KS4}.

By descending induction, $r$ is  Cohen-Lyndon aspherical in  $F_0\,\,\,\,(=F)$.
\end{proof}

\section{Applications}\label{sec:6}

We can combine Theorem~\ref{Thm:KS1}, Theorem~\ref{thm:CL} and Corollary~\ref{Cor:Howie}
to obtain a sufficient condition for a subset of $F$ to be
Cohen\d1Lyndon aspherical in $F$.

\begin{Thm}\label{Thm:big} Let
$(\mathcal{F}, Y)$  be   a graph of groups  with family of groups $(\mathcal{F}(y) \mid y \in Y)$
and family of edge maps $(\overline t_e \colon   \mathcal{F}(e) \to \mathcal{F}(\tau_Y e) \mid e \in EY)$,
let $Y_0$ be a maximal subtree  of~$Y$, and let $F$ be the  fundamental group of
$(\mathcal{F}, Y,Y_0)$.

Let $U$ be a subset of $VY$ and let $(r_v \mid v \in U)$ be a family of elements of $F$ with the property that,
for each $v \in U$,   the following hold.
\begin{enumerate}[{\hskip .6cm}\normalfont(a).]
 \item   $r_v \in \mathcal{F}(v)$ and $\mathcal{F}(v)$ is locally indicable  and
$\mathcal{F}(v)$  acts on some tree with trivial edge stabilizers and no vertex fixed by $r_v$.
 \item For each $e \in \iota_Y^{-1} (\{v\})$, $\mathcal{F}(e)$ is a free factor of
$\mathcal{F}(v)$ that contains no $\mathcal{F}(v)$-conjugate of~$r_v$.
 \item For each $e \in \tau_Y^{-1} (\{v\})$,
$ \lsup{\overline t_e}{\mathcal{F}(e)}$ is a free factor of
$\mathcal{F}(v)$ that contains
no $\mathcal{F}(v)$-conjugate of~$r_v$.
\end{enumerate}
 Then  $(r_v \mid v \in U)$ is Cohen-Lyndon aspherical in $F$. \hfill\qed
\end{Thm}

We shall consider only the case that corresponds to the strongly staggered conditions.

\begin{Cor}\label{Cor:big} Suppose that $F$ is a locally indicable group and that $R$~is
a subset of~$F$.  If there exists some
  $F$-tree $T$ with trivial edge stabilizers such that   $R$ is strongly $T$-staggerable modulo $F$, then $R$ is
Cohen-Lyndon aspherical in   $F$.
\end{Cor}

Cohen and Lyndon~\cite[Corollary~2.2]{CohenLyndon} proved
the case of this result where $F$ is free and $\abs{F\leftmod VT}=1$.

\begin{proof}[Proof of Corollary~\ref{Cor:big}]
Applied to $F$ acting on $T$, the Bass-Serre Structure Theorem presents $F$
as the fundamental group of a certain graph of groups
$(\mathcal{F},Y,Y_0)$.  Let
$<$ be an ordering of $ET$ such that $R$ is strongly $(T,<)$-staggered modulo~$F$.
Here $EY = F\leftmod ET$ and, by the staggered conditions,
the ordering $<$ of $ET$ induces an ordering, again denoted by~$<$, of
 $EY$.
By the strongly staggered conditions,
$(EY, <)$ is order isomorphic to an interval in~$\Z$.
Let us consider only the case where \mbox{$(EY, <)$} is order isomorphic to $\Z$;
the case where \mbox{$(EY, <)$}
 is order isomorphic to a proper interval in $\Z$  is handled in a similar way.
By choosing an order isomorphism we get an  indexing
 $e_\Z$   of the elements of $EY$.

For any interval $I $ in $ \Z$, $e_I$ is an interval in $e_\Z$, and we
 let $Y_I$ denote the subgraph of $Y$ with edge set~$e_I$
together with the vertices of $Y$ that are incident to these edges.  By the
strongly staggered conditions, $Y_I$ is connected and $Y_0 \cap Y_I$ is the unique maximal subtree of~$Y_I$.
Then the graph of groups obtained by restricting $\mathcal{F}$ to $Y_I$ has a well-defined
fundamental group which we denote by $F_I$.  It can be seen that $F_I$ is a
free factor of~$F$.

By the staggered conditions, the ordering $<$ of $EY$ induces an ordering, again denoted by~$<$, of $R$.
Let us consider only the case where $R$ is order isomorphic to $\Z$; the case where
$R$ is order isomorphic to a proper interval in $\Z$  is handled in a similar way.
By choosing an order isomorphism we get an indexing $r_\Z$ of the elements of $R$.
For each $n \in \Z$, we   set
$\mu_n \coloneq \min (F\leftmod (F(\Eaxis(r_n)))) $ and $\nu_n \coloneq \max(F\leftmod (F(\Eaxis(r_n))))$.
 Clearly   $\mu_n \le \nu_n$.  The staggered conditions imply that
 $\mu_n < \mu_{n+1}$ and $\nu_n < \nu_{n+1}$.

We can then  express $F$ as the fundamental group of a second graph of groups whose family of
vertex groups is $(F_{[\mu_n{\uparrow} \nu_n]} \mid n \in \Z)$ and whose family of edge groups is
  $(F_{[\mu_n{\uparrow} \nu_n]} \cap F_{[\mu_{n+1}{\uparrow} \nu_{n+1}]} \mid n \in \Z),$
with the edge maps being those suggested by the intersection notation.  Here the underlying graph has the
form of the real line.

For each $n \in \Z$,  some $F$-conjugate of $r_n$ lies in $ F_{[\mu_n{\uparrow} \nu_n]}$ and we may assume that
$r_n$ itself lies in the vertex group $ F_{[\mu_n{\uparrow} \nu_n]}$.
By the definition of $\mu_n$ and $\nu_n$ and the staggered conditions, no $F$-conjugate of $r_n$
lies in either of the two incident-edge groups.

By Theorem~\ref{Thm:big}, $R$ is Cohen-Lyndon aspherical in $F$.
\end{proof}

Cohen-Lyndon asphericity has many applications, and we  conclude by mentioning two of them.

\begin{Sett}\label{Sett:3} Suppose that $F$ is locally indicable, and that $T$ is an
$F$-tree with trivial edge stabilizers, and that
$R_0$ is a subset of~$F$  which
 is  strongly $T$-staggerable modulo~$F$.

By Definitions~\ref{defs:tree}, $R_0$ has unique roots in $F$ and $\sqrt[F]{R_0}$ is
strongly $T$-staggerable modulo~$F$.
For each $r \in R_0$, we have $\mathbf{C}_F(r) = \gen{\mkern-6mu \sqrt[F]{r}}$.

  By  Corollary~\ref{Cor:big}, $R_0$ and $\sqrt[F]{R_0}$
are Cohen-Lyndon aspherical in $F$.

Let \mbox{$G \coloneq F/\normgen{R_0}$.}
For each $r \in R_0$, let $G_r \coloneq \mathbf{C}_F(r)/\gen{r} = \gp{\sqrt[F]{r}}{r}$,
a finite, cyclic subgroup of $G$.

 Let $(\,Y_r\mid r \in R_0)$ be
 a family of subsets of $F$ with the properties that, for each $r \in R_0$, $Y_r$ is a transversal
for the  $ \gen{\lsup{F}{\mkern-2mu (\sqrt[F]{R_0})}}$-action on~$F$ by multiplication on the left (or right) and
$\mathop{\bigcup}\limits_{r\in R_0} (\lsup{\,Y_r}{(\sqrt[F]{r})})$\vspace{-2mm}  is  a free-generating set of
$\gen{\lsup{F}{\mkern-2mu (\sqrt[F]{R_0})}}$.

 Let $(X_r\mid r \in R_0)$ be a family of subsets of $F$ with the properties that, for each $r \in R_0$, $X_r$ is
a transversal for the  $(\gen{\lsup{F}{\mkern-2mu R}}\gen{\sqrt[F]{r}})$-action on~$F$
 by multiplication on the right  and
$\mathop{\bigcup}\limits_{r\in R_0} (\lsup{X_r\,}{r})$
  is  a free-generating set of~$\gen{\lsup{F}{\mkern-2mu R_0}}$.
For each element $v$ in the $G$-graph $\gen{\lsup{F}{\mkern-2mu R_0}}\leftmod T$,
let $G_v$ denote the $G$-stabilizer of $v$.  Also, let $E_0$, resp. $V_0$, be a
transversal for the  natural $G$-action  on
$\gen{\lsup{F}{\mkern-2mu R_0}}\leftmod ET$, resp.
$\gen{\lsup{F}{\mkern-2mu R_0}}\leftmod VT$.  Here, $F$ is the free product of a
free group and the group $\mathop{\ast}\limits_{v\in V_0} G_v$.
\hfill\qed
\end{Sett}

We first discuss torsion.

\begin{Cor}\label{Cor1}   In {\normalfont Setting~\ref{Sett:3}},
let $N$ denote the  $($normal\,$)$ subgroup of $G$ generated by the elements of $G$ of finite order.
Then the following hold.
\begin{enumerate}[{\normalfont (i).}]
\item For each $r \in R_0$, $Y_r$ may be viewed as a transversal for the
$N$-action on~$G$ by multiplication on the left,
and \mbox{$N
=\mathop{\ast}\limits_{ r \in  R_0}\mathop{\ast}\limits_{y\in Y_r} \lsup{y}{G_r}$,}\vspace{-1.6mm}
and, hence, $N$ is a free product of finite, cyclic groups.
\item $G/N$ is locally indicable, and every torsion-free subgroup of $G$ is locally indicable.
\item Each non-trivial, torsion subgroup of $G$  lies in exactly one of the $($finite, cyclic\,$)$
subgroups of $G$ of the form   $  \lsup{g}{G_r} $,     $r \in R_0$,  $g\in G$,
and  then $r$ is unique and the coset $gG_r \in G/G_r$ is unique.
\end{enumerate}
\end{Cor}

In the case where $F$ acts freely on $T$, and $\abs{R}=1$, (iii) can be attributed to Magnus and Lyndon~\cite{Lyndon},
 (i)~is due to Fischer, Karrass and Solitar~\cite{FKS}, and (ii) can be attributed to  Brodski\u{\i}~\cite{Brodskii84}.

In the case where $F\leftmod T$ has one edge and two vertices, and $\abs{R}=1$,
these results can be attributed  to Edjvet and Howie.

\begin{proof}[Proof of Corollary~\ref{Cor1}]
Let $\overline G \coloneq F/\normgen{\sqrt[F]{R_0}}$.

 By Theorem~\ref{Thm:1}(ii), $\overline G$ is
locally indicable, and, in particular, $\overline G$ is torsion free.

Let $N_1 \coloneq \gen{\lsup{F}{(\sqrt[F]{R_0}})}/\gen{{\lsup{F}{R_0}}}$.
Then $N_1 \trianglelefteq G$ and $G/N_1 = \overline G$.
Since  $\overline G$  is torsion free, every element of $G$ of finite order lies in $N_1$,
that is, $N \le N_1$.
Also, for each $r \in R_0$, the (faithful) image of $Y_r$ in $G$  is a
transversal for the   $N_1$-action on~$G$ by multiplication on the left.

Since $\mathop{\bigcup}\limits_{r\in R_0} (\lsup{(\gen{\lsup{F}{\mkern-2mu (\sqrt[F]{R_0})}}Y_r)}{r})
=\mathop{\bigcup}\limits_{r\in R_0} (\lsup{F}{r})  = \lsup{F}{\mkern-2mu R_0}$, it follows that
$$N_1= \gen{\lsup{F}{\mkern-2mu(\sqrt[F]{R_0}\,\,)}}/\gen{\lsup{F}{\mkern-2mu R_0}}
=\gp{\mathop{\textstyle\bigvee}\limits_{r\in R_0}
(\lsup{Y_r}{(\sqrt[F]{r})})}{\mathop{\textstyle\bigvee}\limits_{r\in R_0} (\lsup{Y_r}{r})}
=\mathop{\ast}\limits_{r\in  R_0} \mathop{\ast}\limits_{y\in Y_r} \lsup{y}{\gp{\sqrt[F]{r}}{r}}
=\mathop{\ast}\limits_{r\in  R_0} \mathop{\ast}\limits_{y\in Y_r} \lsup{y}{G_r}.$$
Thus, $N_1$ is a free product of finite, cyclic groups.
In particular, $N_1$ is generated by some set of elements of finite order in $G$, that is, $N_1 \le N$.
Hence $N_1=N$ and (i) holds.  Also, every torsion-free subgroup of $N$ is free.  Hence (ii) holds.

Any non-trivial, torsion subgroup $H$ of $G$ lies in $N$, and, by well-known properties of free products,
there exists some $n \in N$ such that
$H \le \lsup{\,ny}{G_r}$  for a unique $r\in R_0$ and a unique $y \in Y_r$, and here the coset
$n\lsup{\,y}{G_r} \in N/\lsup{\,y}{G_r}$ is unique.
 Since the (faithful) image of $Y_r$ in $G$ is a transversal for the $N$-action on~$G$
by multiplication on the left,
we see that (iii) holds.
\end{proof}

We now consider exact sequences and  homology groups.

\begin{Cor}\label{Cor2} In  {\normalfont Setting~\ref{Sett:3}},
the following hold.
\begin{enumerate}[{\normalfont (i).}]
\item The left $F$-action  on $\gen{\lsup{F}{\mkern-2mu R_0}}$ by  conjugation
induces a left $G$-action  on $\gen{\lsup{F}{\mkern-2mu R_0}}^{\text{ab}}$, and
 $\gen{\lsup{F}{\mkern-2mu R_0}}^{\text{ab}}$ is then a left $\Z G$-module that is naturally isomorphic to
$\mathop{\bigoplus}\limits_{r\in R_0} \Z[ G/G_r]$.
\item There exist  natural exact sequences of
left $\Z G$-modules that have the form
\begin{equation}\label{eq:1}
0 \to \textstyle \mathop{\bigoplus}\limits_{r \in R_0} \Z[G/G_r] \to \mathop{\bigoplus}\limits_{e\in E_0} \Z[G]\to
\mathop{\bigoplus}\limits_{v\in V_0} \Z[G/G_v]  \to \Z \to 0,
\end{equation}
\begin{equation}\label{eq:2}
0 \to \textstyle \mathop{\bigoplus}\limits_{r \in R_0} \Z[G]
\to \mathop{\bigoplus}\limits_{e \in R_0 \vee E_0} \Z[G]  \to
\mathop{\bigoplus}\limits_{v \in R_0 \vee V_0} \Z[G/G_v]   \to \Z \to 0.
\end{equation}
\item For each $n \in [3{\uparrow}\infty[\,$, the  change-of-groups natural transformation
$\textstyle \mathop{\bigoplus}\limits_{v \in R_0 \vee V_0}
\mathbf{H}_n(G_v,-)  \mkern8mu \to  \mkern8mu \mathbf{H}_n(G,-)  $,
between functors from the category of right  $\Z G$-modules to
the category of abelian groups, is an isomorphism of functors,
and the  change-of-groups natural transformation
$\mathbf{H}^n(G,-)
 \mkern8mu \to  \mkern8mu
\textstyle \mathop{\prod}\limits_{v \in R_0 \vee V_0} \mathbf{H}^n(G_v,-) $,
between functors from the category of left  $\Z G$-modules to
the category of abelian groups,  is an isomorphism of functors.
\end{enumerate}
\end{Cor}

In the case where $F$ is free and $\abs{F\leftmod VT} = 1$, (i) is
Lyndon's  Identity Theorem \mbox{\cite[Section~7]{Lyndon};} in the case where
$F$ acts freely and  $\abs{R_0} = 1$, (i) is Lyndon's Simple Identity Theorem~\cite[Section~7]{Lyndon}.
In the case where $\abs{R_0} = 1$, and $F\leftmod T$ has one edge and two vertices,
(i) is Howie's Simple Identity Theorem~\cite[Theorem~11]{Howie84}.
The  results (ii), (iii)  are straightforward consequences of~(i);
see~\cite[Theorem~11.1]{Lyndon} and~\cite[Theorem~3]{Howie84}.

\begin{proof}[Proof of Corollary~\ref{Cor2}] For each $r \in R_0$, we have bijective correspondences
$$X_r\mkern8mu\simeq\mkern8muF/(\gen{\lsup{F}{\mkern-2mu R_0}}\gen{\sqrt[F]{r}})
\mkern8mu\simeq\mkern8muG/\gen{\sqrt[F]{r}\gen{\lsup{F}{\mkern-2mu R_0}}}\mkern8mu\simeq\mkern8muG/G_r.$$
 Since $\gen{\lsup{F}{\mkern-2mu R_0}}
= \gp{ \mathop{\textstyle\bigvee}\limits_{r\in R_0} (\lsup{X_r}{ r})}{\quad}$, we then have isomorphisms of abelian groups
$$\gen{\lsup{F}{\mkern-2mu R_0}}^{\text{ab}}\mkern8mu\simeq\mkern8mu
\Z[\,\ \mathop{\textstyle\bigvee}\limits_{r\in R_0} (\lsup{X_r}{ r})\,]
\mkern8mu\simeq\mkern8mu \Z[\textstyle \mathop{\bigvee}\limits_{r\in R_0} (X_r)\,]
\mkern8mu\simeq\mkern8mu \Z[\textstyle \mathop{\bigvee}\limits_{r\in R_0} (G/G_r)]
\mkern8mu\simeq\mkern8mu  \mathop{\bigoplus}\limits_{r\in R_0} \Z[ G/G_r].$$
The composite isomorphism of abelian groups $\gen{\lsup{F}{\mkern-2mu R_0}}^{\text{ab}}
\mkern8mu\simeq\mkern8mu  \mathop{\bigoplus}\limits_{r\in R_0} \Z[ G/G_r]$\vspace{-2.6mm}
is compatible with the $G$-actions, and we find that (i) holds.

By Theorem~\ref{Thm:1}(i),   $\gen{\lsup{F}{\mkern-2mu R_0}}$ acts freely on $T$.
By~\cite[Definitions I.8.1]{DicksDunwoody89}, the fundamental group of $\gen{\lsup{F}{\mkern-2mu R_0}}\leftmod T$
can be identified with $\gen{\lsup{F}{\mkern-2mu R_0}}$.
Then (i) and~\cite[Theorem I.9.2]{DicksDunwoody89} give a sequence as in~\eqref{eq:1}.
We leave it as an exercise to  construct  a sequence as in~\eqref{eq:2} which maps
onto the above-constructed sequence in~\eqref{eq:1} with kernel a short exact sequence.
Alternatively, one can arrange for the sequences in~\eqref{eq:1} and~\eqref{eq:2} to be the
 augmented cellular chain complexes of  acyclic, simply connected,
  hence  contractible,
CW-complexes on which $G$ acts by permuting the cells.  Here,
the second CW-complex is obtained from the first CW-complex
by $G$-equivariantly drawing on each two-cell a point and a finite set of edges joining the new point to old points.
This second CW-complex has the property that every non-trivial finite
subgroup of $G$ fixes exactly one point of the space and the fixed point is a zero-cell.

For the short exact sequence
$$0 \to \ker \epsilon \to \textstyle \mathop{\bigoplus}\limits_{v \in R_0 \vee V_0} \Z[G/G_v]
\xrightarrow{\epsilon} \Z \to 0$$
of left $\Z G$-modules  where $\epsilon$ is the corresponding augmentation map, we see from~\eqref{eq:2}  that   $\ker \epsilon$   has
a free $\Z G$-resolution of length at most two.
Resulting long exact sequences in homology then
show that (iii) holds.
\end{proof}

\bigskip

\noindent{\textbf{\Large{Acknowledgments}}}

\medskip
\footnotesize

The research of the first- and second-named authors was
jointly funded by the MCI (Spain) through project MTM2008\d101550.

We are grateful to Ian Chiswell and Jim Howie  for several helpful observations.

\bibliographystyle{amsplain}

\textsc{Departament de  Matem\`atiques,
Universitat Aut\`onoma de Barcelona,
E-08193 Bella\-terra (Barcelona), Spain}

\emph{E-mail address}{:\;\;}\url{yagoap@mat.uab.cat}

\medskip

\textsc{Departament de  Matem\`atiques,
Universitat Aut\`onoma de Barcelona,
E-08193 Bella\-terra (Barcelona), Spain}

\emph{E-mail address}{:\;\;}\url{dicks@mat.uab.cat}

\emph{URL}{:\;\;}\url{http://mat.uab.cat/~dicks/}

\medskip

\textsc{Department of Mathematics,
Virginia Tech,
Blacksburg, VA 24061-0123, USA}

\emph{E-mail address}{:\;\;}\url{linnell@math.vt.edu}

\emph{URL}{:\;\;}\url{http://www.math.vt.edu/people/plinnell/}
\end{document}